\documentclass[12pt,reqno]{amsart}
\usepackage{amsmath,amsxtra,amssymb,amsthm,amsfonts,eufrak,bm}
\usepackage{hyperref}
\usepackage{cleveref}

\usepackage{tikz-cd}
\makeatletter
\newcommand*{\rom}[1]{\expandafter\@slowromancap\romannumeral #1@}
\makeatother
\usetikzlibrary{matrix,arrows,decorations.pathmorphing}

\newcommand{\kl}{\pl \le \pl}

\newcommand{\lel}{\pl = \pl}

\newcommand{\supp}{\operatorname{supp}}
\newcommand{\R}{{\mathbb R}}

\newcommand{\Z}{{\mathbb Z}}
\newcommand{\C}{{\mathbb C}}
\newcommand{\ten}{\otimes}

\newcommand{\pl}{\hspace{.1cm}}

\newcommand{\ran}{\rangle}
\newcommand{\lan}{\langle}
\newcommand{\al}{\alpha}
\newcommand{\si}{\sigma}

\newcommand{\la}{\lambda}
\newcommand{\eps}{\varepsilon}

\newcommand{\id}{\iota_{\infty,2}^n}

\newcommand{\E}{{\mathcal E}}

\newcommand{\A}{{\mathcal A}}

\newcommand{\M}{{\mathcal M}}

\renewcommand{\S}{{\mathcal S}}
\newcommand{\T}{\mathbb{T}}

\newcommand{\N}{{\mathcal N}}

\renewcommand{\o}[1]{\overset{\circ}{#1}}

\newcommand{\norm}[2]{\parallel \! #1 \! \parallel_{#2}}

\newtheorem{lemma}{Lemma}[section]
\newtheorem{prop}[lemma]{Proposition}
\newtheorem{theorem}[lemma]{Theorem}
\newtheorem{cor}[lemma]{Corollary}

\newtheorem{rem}[lemma]{Remark}

\newcommand{\re}{\begin{rem}\rm}
\newcommand{\mar}{\end{rem}}
\newtheorem{exam}[lemma]{Example}
\newcommand{\bra}[1]{\langle{#1}|}
\newcommand{\ket}[1]{|{#1}\rangle}

\newcommand{\qd}{\end{proof}\vspace{0.5ex}}
\newcommand{\prf}{\begin{proof}[\bf Proof:]}

\newcommand{\xspace}{\hbox{\kern-2.5pt}}

\newcommand{\grad}{\text{grad} }

\renewcommand{\id}{\operatorname{id}}
\renewcommand{\supp}{\operatorname{supp}}
\newcommand{\dom}{\operatorname{dom}}
\newcommand{\ARic}{\operatorname{GRic}}
\newcommand{\Ric}{\operatorname{Ric}}

\renewcommand{\R}{\mathcal{R}}

\renewcommand{\id} {\operatorname{id}}

\allowdisplaybreaks
\newtheorem{defi}[lemma]{Definition}
\begin{document}
\title{Complete Logarithmic Sobolev inequalities via Ricci curvature bounded below}
\author{Michael Brannan}
\address{Department of Mathematics\\ Texas A\&M University, College Station, TX 77840, USA} \email[Michael Brannan]{mbrannan@math.tamu.edu}
\author{Li Gao}
\address{Department of Mathematics\\
Texas A\&M University, College Station, TX 77840, USA} \email[Li Gao]{ligao@math.tamu.edu}
\author{Marius Junge}
\address{Department of Mathematics\\
University of Illinois, Urbana, IL 61801, USA} \email[Marius Junge]{mjunge@illinois.edu}

\begin{abstract} We prove that for a symmetric Markov semigroup, Ricci curvature bounded from below by a non-positive constant combined with a finite $L_\infty$-mixing time implies the modified log-Sobolev inequality.  Such $L_\infty$-mixing time estimates always hold for Markov semigroups that have spectral gap and finite Varopoulos dimension. Our results apply to non-ergodic quantum Markov semigroups with noncommutative Ricci curvature bounds recently introduced by Carlen and Maas. As an application, we prove that the heat semigroup on a compact Riemannian manifold admits a uniform modified log-Sobolev inequality for all its matrix-valued extensions.
\end{abstract}
\maketitle

\section{Introduction}
In differential geometry, Ricci curvature lower bounds have many applications in topology, geometry and analysis. One pioneering work that connects Ricci curvature with analysis of heat semigroups is the Bakry-Emery theorem \cite{BE85}. It implies that if the Ricci curvature of a compact Riemannian manifold $(M,g)$ is bounded from below by a positive constant, then the heat semigroup satisfies a logarithmic Sobolev inequality. In this paper, motivated by quantum information theory, we present a uniform approach to obtain logarithmic Sobolev inequalities from a non-positive Ricci curvature lower bound for both classical and quantum Markov semigroups. Indeed, we show that a non-positive Ricci curvature lower bound plus a $L_\infty$-time to equilibrium implies logarithmic Sobolev inequality in the noncommutative non-ergodic setting.

In the past decades, the notion of Ricci curvature lower bound has been largely extended beyond Riemannain manifolds using ideas from optimal transport.
Motivated by Gromov's Precompactness theorem \cite{gromov81}, Lott-Villani \cite{LV09} and Strum \cite{sturm} independently introduced a notion of Ricci curvature lower bound for metric measures spaces. Such a space has Ricci curvature bounded below by a constant $\la$ if the entropy, as a functional on the state space (space of probability measures), is $\la$-convex along geodesics of the $L_2$-Wasserstein distance. 
Later, similar ideas were extended to Markov semigroups on discrete spaces and noncommutative spaces. The key ingredient is to construct an analog of the Wasserstein distance $W$ on the state space such that the semigroup is the gradient flow of the entropy functional with respect to $W$. Such gradient flow constructions were obtained independently in \cite{Maas1,mielke11,Chow12} for Markov process on finite state spaces, and \cite{CM12,CM,Mielke13,mielke17} for finite dimensional quantum systems. More recently, the noncommutative Wasserstein metric has been further studied on finite von Neumann algebras \cite{wirth,hornshaw}. Based on these, the notions of Ricci curvature lower bound via $\la$-convexity of entropy has been studied by Erbar-Maas \cite{ME11} for discrete spaces and by Carlen-Maas \cite{CM18}, Datta-Rous\'e \cite{DR} and Wirth \cite{wirth} for noncommuative spaces. Thanks to the gradient flow structure, the connection between Ricci curvature and functional inequalities, including the extensions of the Bakry-Emery theorem, have been obtained in all the above settings.

The logarithmic Sobolev inequalities were first introduced by Gross \cite{Gross75a,Gross75b} as a reformulation of hypercontractivity, and have been intensively studied since then (see \cite{Gross14} for an overview). The focus of this paper is the $L_1$-version of the log-Sobolev inequality, also called the modified log-Sobolev inequality. Indeed, let $T_t=e^{-At}:L_{\infty}(\Omega,\mu)\to L_{\infty}(\Omega,\mu)$ be a Markov semigroup with Dirichlet form $\E(f)=(f,Af)$. We say $T_t$ satisfies a $\la$-modified log-Sobolev inequality ($\la$-MLSI) if for any probability density function $f$,
\[ 2\la \int  f\log f d\mu\le \E (f,\log f)\pl, \pl \forall\pl f\ge 0, \int fd\mu=1\]
The integral on the left hand side of the above inequality is the entropy $H(f)=\int  f\log f d\mu$ and the right hand side is called the Fisher information $I(f)=\int (A f)\log f d\mu$, which describes the rate of decrease of entropy: $I(T_tf)=-\frac{d}{dt}H(T_t(f))$. Intuitively, MLSI characterizes the exponential decay of entropy along the time evolution of the semigroup. In the smooth setting, MLSI is equivalent to the more common $L_2$-log-Sobolev inequality
\begin{align}\la \int  g^2\log g^2 d\mu\le 2\E (g,g)\pl, \pl \forall\pl g\ge 0, \int g^2d\mu=1 \label{LSII}\pl.\end{align}
However, it is weaker than \eqref{LSII} in discrete and noncommutative cases. See \cite{Ledoux} for a review article on the interplay between spectral gap, log-Sobolev inequalities and Ricci curvature. More recently, Otto-Villani \cite{OV} proved that MLSI also implies Talagrand's transport cost inequality, which further bounds spectral gap and derives concentration of measure phenomena. Recently these application of MLSI has also been extended to (finite dimensional) quantum Markov semigroups \cite{CM,DR19}, which suggest a uniform picture of functional inequalities for both classical and noncommutative settings.

Quantum Markov semigroups are noncommutative generalization of classical Markov semigroups, where the underlying function space is replaced by matrix algebras or operators algebras. A quantum Markov semigroup on a von Neumann algebra $\M$ is an ultra-weakly continuous family $(T_t)_{t\ge 0}:\M\to \M$ of normal unital completely positive maps. When $\M=B(H)$ is the bounded operators on a Hilbert space $H$, quantum Markov semigroups models the time evolution of dissipative open quantum system. In operator algebras, quantum Markov semigroups have been widely studied in the context of approximation properties, structure theory, and noncommutative harmonic analysis (see e.g. \cite{Caspers18,jmp}). In this paper, we will focus on \emph{symmetric quantum Markov semigroup} on finite von Neumann algebras. That is, $\M$ is a von Neumann algebra equipped with a normal faithful tracial state $\tau$,  and the semigroup $T_t:\M\to \M$ is given by self-adjoint maps with respect to the $\tau$-inner product. This setting avoids the techicalities of Tomita–Takesaki theory, but is still broad enough to cover many examples of wide interest, such as classical Markov semigroups on probability spaces, finite dimensional dissipative systems in quantum information theory, and also various infinite dimensional examples in operator algebras.

One of the main motivations for this work is to prove a MLSI for quantum Markov semigroups that is stable under tensor products. For classical Markov semigroups, it is known that if a pair of semigroups $S_t,T_t$ satisfy $\la$-MLSI, then $S_t\ten T_t$ satisfies $\la$-MLSI. Tensorization is a useful property that allows us to obtain MLSI for composite systems by studying smaller, more tractable subsystems. In the noncommutative setting, tensor stability of MLSI generally requires not only MLSI but a ``completely bounded'' version of MLSI: $T_t$ is said to satisfy a $\la$-complete log-Sobolev inequality ($\la$-CLSI) if all of its  matrix-valued extensions $T_t\ten \id_{M_n}$ satisfy $\la$-MLSI. For quantum Markov semigroups, CLSI has the tensor-stability  property that $S_t$ and $T_t$ satisfy $\la$-CLSI $\Rightarrow$ $S_t\ten T_t$ satisfies $\la$-CLSI \cite{CLSI}. For classical Markov semigroups, CLSI simply means an uniform MLSI constant for all matrix-valued functions, and for quantum Markov semigroups, CLSI has applications in estimating decay rates of entanglement. The study of CLSI naturally leads us to consider non-ergodic semigroups, because $T_t\ten \id$ always has non-trivial fixed-point space.

We now describe the content of paper and state our main results. Section 2 reviews the basic definitions and proves some preliminary lemmas.

The main theorem of this paper is discussed in Section 3, which we illustrate here using the example of the heat semigroup. Let $T_t=e^{-\Delta t}$ be the heat semigroup on a compact manifold $(M,g)$. There are two key ingredients in our proof. The first one is (displaced) monotonicity of Fisher information. The idea goes back to the Bakry-Emery theorem, in the proof of which they actually showed the implications
\begin{align} \Big\{\text{Ricci {curvature lower bound }} \lambda \Big\} \pl \overset{\la\in \mathbb{\mathbb{R}}}{\Longrightarrow}\pl \Big\{I(T_tf)\le e^{-2\la t} I(f) \pl \forall t \ge 0 \Big\}\pl \overset{\la>0}{\Longrightarrow}  \la\text{-MLSI}. \label{chain}\end{align}
We call the middle inequality ``$\la$-Fisher monotonicity", as for $\la=0$, it asserts that $I(T_tf)$ is non-increasing in $t$. For $\la>0$, this immediately implies $\la$-MLSI. For $\la\le 0$, we will need a second ingredient, which is the finiteness of the following $L_\infty$-mixing time
\[t_{cb}=\inf\{t >0| \norm{T_t-E:L_1(M,d\mu)\to L_\infty(M,d\mu)}{}\le 1/2\}<\infty\]
Here $E(f)=(\int fd\mu)1$ is the averaging map. We prove that this $L_\infty$-mixing time is the half-decay time for entropy $H(T_tf)$, and $t_{cb}$ is always finite by the spectral gap of $\Delta$ and standard heat kernel estimates. All the notions mentioned above including the implication \eqref{chain} are fully adapted to the noncommutative non-ergodic setting, which leads to the statement of our main theorem.
\begin{theorem}[c.f. Theorem \ref{CLSI1}]\label{mainthm}
Let $T_t:\M\to \M$ be a symmetric quantum Markov semigroup and $E:\M\to \N$ be the conditional expectation onto its fixed point algebra $\N$. Suppose \begin{enumerate}\item[i)]$T_t$ satisfies $\la$-Fisher monotonicity for some $\lambda\in \mathbb{R}$: for all densities $\rho$,
\[I(T_t(\rho))\le e^{-\la t}I(\rho)\pl, \forall t\ge 0\]
\item[ii)] $T_t$ has finite completely bounded return time: \[t_{cb}=\inf\{t >0 \pl | \norm{T_t-E:L_\infty^1(\N\subset \M)\to L_\infty(\M)}{cb}\le 1/2\}<\infty\pl.\]
\end{enumerate}
Then $T_t$-satisfies $\kappa(\la,t_{cb})$-MSLI for
$\kappa(\la, t)=\frac{\la}{2(1-e^{-2t\la})}$
        \end{theorem}
For classical Markov semigroups, it is well-known that the $L_\infty$-mixing time itself implies the log-Sobolev inequality (see \cite{DSC}). Nevertheless, this standard approach via hypercontractivity does not apply to the matrix-valued setting because the famous Rothaus Lemma as a crucial step is no longer valid. We emphasis that our main theorem, using ideas from quantum information theory,
applies to fully non-ergodic noncommutative setting. It allows one to derive MSLI for matrix-valued functions or endomorphism maps on vector bundle, and also the tensor-stable CLSI for quantum Markov semigroups.

In Section 4 we apply the main theorem to various examples in both the classical and quantum contexts. Section 4.1 discusses the connection to Bakry-Emery's curvature dimension condition for Markov diffusion semigroups. An important class of such semigroups are heat semigroups on (weighted) Riemannian manifolds. For heat semigroups, we have the following result
\begin{theorem} [c.f. Theorem \ref{riemann}]
Every heat semigroup on a connected compact (weighted) Riemannian manifold satisfies CLSI.
\end{theorem}

\noindent In Section \ref{liegroup}, we show that any ``central'' semigroup on a compact group has entropy curvature bound zero, and based on that, we estimate the optimal CLSI constant for the heat semigroup on $d$-torus $\mathbb T^d$. For noncommutative examples, Section \ref{depolar} studies entropy Ricci curvature bounds and MLSI constants for depolarizing semigroups. We also consider Schur multiplier semigroups and semigroups of random unitary channels in Section \ref{schur} \& \ref{randomu}. We end our paper discussion with an appendix on approximations of relative entropy.

\subsection*{Acknowledgements} Li Gao thanks Haonan Zhang for helpful discussions on Proposition \ref{fisher}. We thank Melchior Wirth for pointing out a previous mistake on Proposition \ref{depo}.
 Michael Brannan was partially supported by NSF Grants DMS-2000331 and  DMS-1700267.   Marius Junge was partially supported by NSF grants DMS-1839177 and  DMS-1800872.

\section{Preliminaries}
\subsection{Entropy and Relative Entropy}Throughout the paper, we let $\M$ be a finite von Neumann algebra equipped with a normal faithful finite tracial state $\tau$. For $0<p<\infty$, the $L_p$-space $L_p(\M)$ is defined as the completion of $\M$ with respect to the norm
\[\norm{a}{p}=\tau(|a|^{p})^{1/p}\pl.\]
 We identify $L_\infty(\M):=\M$ and the predual space $\M_*\cong L_1(\M)$ via the duality
\[a\in L_1(\M)\longleftrightarrow \phi_a\in \M_*,\pl  \phi_a(x)=\tau(ax)\pl.\]
We say $\rho \in L_1(\M)$ is a density operator (or simply density) if $\rho\ge 0$ and $\tau(\rho)=1$. The set of all densities correspond to the normal states of $\M$, which we denote by $S(\M)$. Throughout the paper, states always mean normal states and are identified with their density operators.

Recall that for two normal positive linear functionals $\rho$ and $\si$, the Umegaki relative entropy is \begin{align*}D(\rho||\si)=\begin{cases}
                             \bra{\rho^{1/2}} \log \Delta(\rho,\si)\ket{\rho^{1/2}}, & \mbox{if } \supp(\rho)\le \supp(\si) \\
                             +\infty, & \mbox{otherwise}.
                           \end{cases}
\end{align*}
where $\Delta(\rho,\si)(x)=\rho x\si^{-1}$ is the relative modular operator and $\ket{\rho^{1/2}}$ is the vector of $\rho^{1/2}$ in $L_2(\M)$. In the tracial setting
\[D(\rho||\si)=\tau(\rho \log \rho-\rho\log \si)\pl,\]
provided $\rho \log \rho, \rho\log \si\in L_1(\M)$. The entropy of $\rho$ is then given by $H(\rho)=D(\rho||1)$. (Note that $H$ is actually the Boltzmann $H$-function, which differs with the usual entropy in information theory by a negative sign). We say
a linear map $\Phi:L_1(\M) \to L_1(\M)$ is completely positive trace preserving (CPTP) if its adjoint $\Phi^\dag:\M\to \M$ is normal, unital, and completely positive (UCP). The monotonicity of the relative entropy under CPTP maps (also called the {\it data processing inequality}) states that for any CPTP $\Phi$ and any two states $\rho,\si$,
\[D(\rho||\si)\ge D(\Phi(\rho)||\Phi(\si))\pl.\]
In particular, we have $D(\rho||\si)\ge 0$ for any $\rho$ and $\si$, and the equality $D(\rho||\si)=0$ holds if and only if $\rho=\si$.

Let $\N\subset \M$ be a von Neumann subalgebra. The conditional expectation $E:\M\to \N$ on to $\N$ is the (unique) completely positive unital and trace preserving map determined by
\[ \tau(xy)=\tau(xE(y)), \forall x\in \N,y\in \M\pl. \]
$E$ is normal and its pre-adjoint map gives an embedding $L_1(\N)\subset L_1(\M)$.
For a state $\rho$, the relative entropy with respect to $\N$ is defined as follows
\[D(\rho||\N):=\inf_{\si\in S(\N)} D(\rho||\si)=D(\rho||E(\rho))\pl.\]
where the infimum is always attained by $E(\rho)$. Indeed, we have the identity that $\si\in \S(\N)$
\[D(\rho||\si)=D(\rho||E(\rho))+D(E(\rho)||\si)\pl,\]
and the infimum is attained if and only if $D(E(\rho)||\si)$ is zero.
If $H(\rho)=D(\rho||1)<
\infty$ is finite, so does
\[H(E(\rho))=D(E(\rho)||1)\le D(\rho||1)=H(\rho)<\infty\]
and \[D(\rho|| \N)=\tau(\rho\log \rho-\rho\log E(\rho))=\tau(\rho\log \rho)-\tau(E(\rho)\log E(\rho))=H(\rho)-H(E(\rho))\pl.\]
If $\Phi$ is CPTP and $\Phi(L_1(\N))\subset L_1(\N)$ (or equivalently $\Phi^\dag(\N)\subset\N$), we have
the data processing inequality for $D(\rho||\N)$,
\[D(\Phi(\rho)||\N)\le D(\Phi(\rho)||\Phi\circ E(\rho))\le D(\rho||E(\rho))=D(\rho||\N)\pl.\]
Here the second inequality follows from $\Phi\circ E(\rho)\in S(\N)$.
As already seen in \cite{bardet,CLSI}, the relative entropy $D(\rho||\N)$ is crucial in functional inequalities for non-ergodic Markov semigroups.

\subsection{Quantum Markov Semigroups}
A quantum Markov semigroup is a family of linear maps $\displaystyle (T_t)_{t\ge 0}:\M\to \M$ with the following properties
\begin{enumerate}
\item[i)] $T_t$ is a normal UCP map for all $t\ge 0$.
\item[ii)] $T_t\circ T_s=T_{s+t}$ for any $t,s\ge 0$ and $T_0=\id$.
\item[iii)] for each $x\in \M$, $t\mapsto T_t(x)$ is continuous in ultra-weak topology.
\end{enumerate}
The generator of the semigroup is defined as
\[\pl Ax=w^*-\lim_{t\to 0} \frac{x-T_t(x)}{t}\pl, \pl T_t=e^{-At}\pl,\]
where $A$ is a closable densely defined operator on $L_2(\M)$. We say a quantum Markov semigroup $(T_t)$ is \emph{symmetric} if
for any $t$, $T_t$ is a self-adjoint map for the $\tau$-inner product,
\[\tau(x^*T_t(y))=\tau(T_t(x)^*y)\pl , \pl  x,y\in \M.\]
We refer to \cite{DL92} for the basic properties of symmetric quantum Markov semigroups.
A symmetric quantum Markov semigroup is determined by its
{\it  Dirichlet form} \[\E:L_2(\M)\to [0,\infty]\pl ,\pl  \E(x,x)=\tau(x^*A x)\pl.\]
We write $\dom (A)$ for the domain of $A$ and $\dom (A^{1/2})$ for the domain of $\E$. The Dirichlet subalgebra $\A_\E:=\dom (A^{1/2})\cap \M$ is a dense $*$-subalgebra of $\M$ and a core of $A^{1/2}$ \cite{DL92}. For symmetric semigroups, $T_t=T_t^\dag$ are unital completely positive and trace preserving (in short, UCPTP), and the generator $A$ is self-adjoint and positive.
Let $\N$ be the common multiplicative domain for $(T_t)$, defined as follows
\begin{align}\N=\{ a\in \M \pl |\pl  T_t(a^*)T_t(a)=T_t(a^*a) \pl\text{and}\pl T_t(a)T_t(a^*)=T_t(aa^*) \pl, \forall \pl t\ge 0\}\label{domain}\end{align}
Let $E$ be the conditional expectation onto $\N$.  For symmetric $(T_t)$, we have
\[T_t\circ E= E\circ T_t=E\pl.\]
Then ${\N=\{x\in \M\pl | \pl T_t(x)=x, \forall t\}}$ is the fixed-point subalgebra, and each $T_t$ is an $\N$-bimodule map,
\[T_t(axb)=aT_t(x)b\pl, \pl  \forall\pl a,b\in \N ,x\in \M\]
In particular, we have $A(\N)=0$ and $\N\subset \A_\E$.

We say $(T_t)$ is \emph{ergodic} if $\N=\C 1$ is trivial. This means the semigroup admits an unique invariant state. We specify the conditional expectation onto the scalars $\mathbb{C}1$ as
$E_\tau(\rho)=\tau(\rho)1$. Throughout the paper, we will focus on symmetric quantum Markov semigroups that are not necessarily \emph{ergodic}. Recall that the {\it gradient form} (or \emph{carr\'e du champ}) of the generator $A$ is the operator given by
\begin{align}
\label{gd} \Gamma(x,y)=\frac{1}{2}\Big((Ax^*)y+x^*Ay-A(x^*y)\Big)\pl.\end{align}
$\Gamma$ is a (completely) positive sesquilinear form because
\[\Gamma(x,x)=\lim_{t\to 0} \frac{1}{t}(T_t(x^*x)-T_t(x^*)T_t(x))\pl,\]
where the right hand side is always positive by the Kadison-Schwarz inequality for unital completely positive maps.  We recall the following fundamental Markov dilation result from the preprint \cite{JRS}. 
\begin{theorem}[\cite{JRS}]\label{JRS1} Let $T_t=e^{-At}:\M\to \M$ be a symmetric quantum Markov semigroup. Suppose $\Gamma(x,x)\in L_1(\M)$ for all $x\in {\rm dom}(A^{1/2})$. Then there exists a trace-preserving embedding $\M \subseteq (\hat \M, \tau)$ into a finite von Neumann algebra $\hat{\M}$, and a closed {\it symmetric derivation} $\delta:{\rm dom}(A^{1/2})\to L_2(\hat{\M})$, meaning that
\begin{enumerate}
\item[i)]$\delta:{\rm dom}(A^{1/2})\to L_2(\hat{\M})$ is a closed linear map such that $\delta(x^*)=\delta(x)^*$.
\item[ii)] $\delta$ satisfies the Leibniz rule: for any $a,b\in \dom(A^{1/2})\cap \M$,
\[\delta(ab)=\delta(a)b+a\delta(b) \pl. 
\]
\end{enumerate}
Moreover, the gradient form $\Gamma$ and the derivation $\delta$ are related through
\begin{enumerate}
\item[iii)] for all $z\in \M$, \begin{align} \label{gradient} \tau(\Gamma(x,y)z) \lel \hat{\tau}(\delta(x)^*\delta(y)z) \pl.\end{align}
Equivalently, $E_\M(\delta(x)^*\delta(y))=\Gamma(x,y)$ where $E_\M:\hat{\M} \to \M$ is the conditional expectation. As a consequence, $A=\delta^*\delta$ as an operator on $L_2(\M)$.
\end{enumerate}
\end{theorem}
The construction of the derivation in Theorem \ref{JRS1} is stronger than the representation theorem for completely Dirichlet forms by Cipriani
and Sauvageot \cite{CS03}. Instead of having a larger von Neumann algebra $\hat{\M}$,  \cite[Theorems 8.2 \& 8.3]{CS03} ensures the existence of a closed derivation $\partial:\dom(A^{1/2})\to H$ into a Hilbert $\M$-bimodule. The derivation $\partial$ satisfies the Leibniz rule with respect to the bimodule action and
\begin{align*}\tau(\Gamma(x,y)z)=\lan z\partial(x),\partial(y)\ran_H \pl,  \forall, z\in \M, x,y\in \dom(A^{1/2})\end{align*}
which is analogous to the property \eqref{gradient}.
The derivation construction in this setting is used in \cite{wirth} and \cite{hornshaw} to construct the noncommutative Wasserstein distance. Throughout the paper, we will focus on symmetric quantum Markov semigroups in order to ensure the existence of the derivation $\delta$ in Theorem \ref{JRS1}, making heavy use of \eqref{gradient} and also the von Neumann algebra structure of $\hat{\M}$.  These ideas are close to the works \cite{CM,CM18} by Carlen and Maas (and also \cite{DR}). Nevertheless, our setting using Theorem \ref{JRS1} is a special case of \cite[Theorem 8.2 \& 8.3]{CS03}, which enables us to apply the results from \cite{CS03} and \cite{wirth}. We recall the following definition from \cite{JLLR}.


\begin{defi}\label{pair}We say $(\A,\hat{\M},\delta)$ is a derivation triple for  $T_t:\M\to \M$ if \begin{enumerate}\item[i)]$(\delta,\hat{\M})$ satisfies properties i)-iii) in the Theorem \ref{JRS1}
\item[ii)] $\A\subset \M$ is a $w^*$-dense subalgebra such that
$ \A\subset \dom(A^{1/2})\pl, T_t(\A)\subset \A$.
\end{enumerate}
\end{defi}
Note that Dirichlet subalgebra $\A_\E=\dom (A^{1/2})\cap \M$ always  satisfies ii). Then it is guaranteed by Theorem \ref{JRS1} that derivation triples always exist for symmetric semigroups. It was proved in \cite[Lemma 7.2]{CS03} that $\A_\E$ is closed under $C^1$-functional calculus. Indeed, let $x\in \M$ be self-adjoint with spectrum ${\rm spec}(x)\subset (a,b)$ and
let $f:(a,b)\to \mathbb{R}$ be a function with continuous bounded derivative. We have
$f(x)\in \A_\E$ and its gradient is given by the double operator integral,
\[ \delta(f(x))=J_F^x(\delta(x)):= \int_{\mathbb{R}}\int_{\mathbb{R}}F(x,y) dE_s \delta(x) d{E}_t\]
  where $E_s$ is spectral projection of $x$ and $F$ is the bi-variable function
\[F: \mathbb{R}\times \mathbb{R} \to \mathbb{R}\pl,  F(s,t)=\begin{cases}
           \frac{f(s)-f(t)}{s-t}, & \mbox{if } s\neq t\\
           f'(s), & \mbox{if } s=t.
         \end{cases}\pl.\]

For concrete examples, it maybe more convenient to work with some smaller algebra $\A\subset \A_\E$ usually with strong regularity. Indeed, for most of examples in our discussions, the derivation triple $(\A,\hat{\M},\delta)$ will be concretely described. In general, by assumption $\A\subset \A_\E$ always holds. Thus the $C^1$-functional calculus is also applicable for $\A$ (with $f(x)$ in $\A_\E$). It follows from Kaplansky density theorem (c.f. \cite[Theorem II.4.8]{takesaki}) that $\A$ is norm dense in $L_1(\M)$ and $L_2(\M)$. Moreover, denote $\A_0=\cup_{t>0} T_t(\A)$. Then $\A_0\subset \dom(A)$ is w$^*$-dense in $\M$ and norm-dense in $L_p(\M)$ for all $1\le p<\infty$ (see \cite[Proposition 2.14 \& 3.1]{DL92}.)


\subsection{Modified logarithmic Sobolev inequalities}
Let $T_t=e^{-At}:\M\to \M$ be a symmetric quantum Markov semigroup and let $(\A,\hat{\M},\delta)$ be a derivation triple of $T_t$. We first specify some subsets of states space.
\begin{align*}
&S_H(\M)=\{\rho\in S(\M) \pl | \pl  H(\rho)<\infty\}\pl, \\
&S_B(\M)=\{\rho\in S(\M) \pl | \pl \la 1\le \rho \le \mu 1\pl, \pl \text{for some}\pl \la,\mu >0\}\\
&S_B(\A_0)=S_B(\M)\cap \A_0\pl.
\end{align*}
Here $S_H(\M)$ are states with finite entropy, $S_B(\M)$ are states with bounded invertible density and $S_B(\A_0)$ are bounded invertible densities in $\A_0=\bigcup_{t>0} T_t(\A)$. Are the three are norm-dense subset of the state space $S(\M)$. Recall that the {\it Fisher information} for $\rho\in S_{B}(\A_0)$ is defined as
\[ I(\rho):=\tau \big((A\rho)\log \rho\big)\]
\begin{defi}\label{defMLSI}
We say a quantum Markov semigroup $T_t=e^{-At}$ satisfies the $\la$-modified logarithmic Sobolev inequality (in short, $\la$-MLSI) for $\la>0$ if
\[ 2\la D(\rho||\N)\le I(\rho)\pl, \pl \forall \rho \in S_B(\A_0)\]
\end{defi}
Note that we have the constant $2$ in the definition to match with curvature constant introduced later.
The definition of Fisher information and the derivative relation \eqref{deri} can be further extended to $\rho \in \dom (A^{1/2})$ as
\[I(\rho):=\lim_{n\to \infty} \E(\rho, \log_{(n)} \rho)\]
where $\log_{(n)}$ is the function $\log_{(n)}(x)=\log (x+e^{-n}) \wedge n$. See \cite[Definition 5.17 \& Proposition 5.23]{wirth}. Nevertheless, it suffices (is more convenient) to consider $\rho \in S_B(\A_0)$ for MLSI.
\begin{prop}
A semigroup $T_t$ satisfies $\la$-MLSI if and only if
\[ D(T_t(\rho)||\N)\le e^{-2\la t} D(\rho||\N)\pl, \pl  \forall \pl \rho \in S(\M).\]
\end{prop}
The proof of the above proposition is a standard density argument included in Appendix and here we illustrate the heuristic. The Fisher information is the negative derivative of (relative) entropy along the semigroup flow
\begin{align}I(\rho)=-\frac{d}{dt}D(T_t(\rho)||\N)|_{t=0}=-\frac{d}{dt}H(T_t(\rho))|_{t=0}\pl. \label{deri}\end{align}
where the second equality follows from \[D(T_t(\rho)||\N)=D(T_t(\rho)||E(\rho))=H(T_t(\rho))-H(E(\rho))\pl.\] In particular, we have $I(\rho)\ge 0$ by the data processing inequality $D(T_t(\rho)||\N)\le D(\rho||\N)$.
Then by Gronwall's Lemma, MLSI is equivalent to exponential decay of relative entropy (see \cite{CLSI,bardet})
\begin{align}D(T_t(\rho)||\N)\le e^{-\la t} D(\rho||\N)\label{entropydecay}\pl, \forall \rho \in  S_{B}(\A).\end{align}
The intuition here is that for non-ergodic semigroups, the semigroup flow $T_t(\rho)$ for an initial state $\rho$ does not converge to one unique equilibrium state, but to its conditional expectation $E(\rho)$. Thus only the relative entropy $D(T_t(\rho)||\N)=D(T_t(\rho)||E(\rho))$ decay to $0$, and the entropy $H(T_t(\rho))=D(T_t(\rho)||1)$ does not converges to $0$. Based on the non-ergodic MLSI, we introduce the complete bounded version of MLSI.\begin{defi}
We say $(T_t)_{t\ge 0}$ satisfies $\la$-complete logarithmic Sobolev inequality ($\la$-CLSI) if $\id_{\mathcal{R}}\ten T_t$ satisfy $\la$-MLSI for any finite von Neumann algebra $\mathcal{R}$.
\end{defi}

Note that CLSI was studied in \cite{CLSI} under the definition that $\id_{M_n}\ten T_t$ satisfy $\la$-MLSI for every matrix algebra $M_n$. Here in this paper, we will work with the stronger definition that $\mathcal{R}$ can be any finite von Neumann algebra. The MLSI is a $L_1$-version of the Gross' logarithmic Sobolev inequality that is usually stated for $L_2$-elements. For an ergodic symmetric Markov semigroup $T_t$, $T_t$ is said to satisfies $\la$-logarithmic Sobolev inequality ($\la$-LSI) if for any positive $x\in \dom(A^{1/2})$ with $\norm{x}{2}=1$,
\[ \la H(x^2)\le 2\E(x,x)\pl.\]
It was proved in \cite[Section III.A.1]{KT13} that all (finite dimensional) symmetric quantum Markov semigroup satisfies strong $L_1$-regularity: $4\E(\rho^{1/2},\rho^{1/2})\le I(\rho)$. Thus we have $\la$-LSI $\Longrightarrow$ $\la$-MLSI for ergodic symmetric Markov semigroups. On the other hand, it was pointed out in \cite[Section 7.4]{CLSI} and \cite[Theorem 5.1]{BD18} that for non-ergodic cases, LSI does not holds for the basic example such as $A=I-E$. This suggests that LSI may not holds for many non-ergodic cases and hence neither the complete version, in contrast to MLSI and its complete version CLSI (see
\cite[Section 5]{CLSI} for a density result).

\subsection{Noncommutative Wassersetin Distance} Let $T_t:\M\to \M$ be a symmetric quantum Markov semigroup and $(\A,\hat{\M},\delta)$ be a derivation triple for $T_t$. For simplicity of notation, we write $\tau$ for the trace on both $\M$ and $\hat{\M}$.
For a state $\rho\in S(\M)$, define the operator
\[[\rho]x:=\int_{0}^1 \rho^{s} x\rho^{1-s} ds=R_\rho\circ f(\Delta_\rho) (x)\pl. \]
Here $R_\rho$ (resp. $L_\rho$) is the right (resp. left) multiplication operator and $\Delta_\rho=L_\rho R_\rho^{-1}$ is the modular operator of $\rho$.
$f(\Delta_\rho)$ is the functional calculus of $\Delta_\rho$ for the function $f(w)=\int^{1}_0\omega^{s}ds=(w-1)/\log w$. The inverse operator (on the support of $\rho$) is
\[ [\rho]^{-1}x= R_\rho^{-1}\circ \frac{1}{f}(\Delta_\rho)x=J_{\log}^\rho(x)=\int_0^\infty (\rho+s)^{-1}x(\rho+s)^{-1} ds,\]
where $J_{\log}^\rho$ is the double operator integral for the function $f(t)=\log t$ and operator $\rho$. The last equality follows from $\frac{\ln x-\ln y}{x-y}=\int_0^\infty (x+s)^{-1}(y+s)^{-1}ds$.
We define the {\it weighted $L_2$-(semi)norm} on $\hat{\M}$ by
 \[ \langle \xi,\eta\rangle_{\rho}:= \langle \xi,[\rho]\eta\rangle_{L_2(\hat{\M},\tau)}
 \lel  \int_0^1  \tau(\xi^* \rho^{1-s}\eta \rho^s) ds \pl .\]
Denote $\hat{\mathcal{H}}_\rho\subset L_2(\hat{\M},\rho)$ as the closure of $\delta(\A_\E)$. Let $I$ be an interval. Following \cite{wirth}, we say a curve $\gamma: (a,b)\to S(\M)$ is {\it admissible} if
\begin{enumerate}
\item[i)] for any $a\in \A$, $s\mapsto \tau(a\gamma(s))$ is  locally absolutely continuous.
\item[ii)] there exists $\xi \in L_{loc}^2((a,b), \hat{\mathcal{H}}_{\gamma(t)})$ such that
\begin{align} \frac{d}{ds}\tau(a\gamma(s))=\lan \delta a, \xi(s) \ran_\rho\pl,\pl  a.e. \pl s\in (a,b) \label{coneq}\end{align}
\noindent Such $\xi$ is unique since  $\delta(\A)$ is dense in $\hat{\mathcal{H}}_\rho$ and we write this as $\xi(s)=D\gamma(s)$.
\end{enumerate}
\begin{defi}For $\rho,\si\in S(\M)$, the noncommutative Wasserstein distance is defined as
\[ W(\rho,\si)=\inf_\gamma \pl \int_0^1\norm{D\gamma(s)}{\gamma(s)}ds  \]
where the infimum is taken over all admissible curves $\gamma:[0,1]\to S(\M)$ such that $\gamma(0)=\rho,\gamma(1)=\si$.
\end{defi}

We say an admissible curve $\gamma:[0,1]\to (S(\M),W)$ is a {\it geodesic} if $\gamma$ attains the infimum of $W(\gamma(0),\gamma(1))$.  We say that $\gamma$ is a {\it geodesic with constant speed} if $W(\gamma(s),\gamma(t))=|s-t|W(\gamma(0),\gamma(1))$.
It was proved in \cite[Lemma 4.19]{wirth} that under the assumption that the smooth subalgebra $\A$ is dense and $L_1(\M)$ is separable, then the infimum above can be taken to be over smooth curves.

 For simplicity, we now illustrate the Riemannian metric for smooth curves on $S_B(\M)$ as in \cite{CM}. The Wasserstein distance induces a pseudo-metric on $S_B(\M)$: for $z\in \M$,
 \[ \|z\|_{g,{\rho}}
 :\lel \inf\{  \pl \|\xi\|_{\rho} \pl | \pl
 \delta^*([\rho]\xi )=z \} \pl .\]
where $\delta^*$ is the adjoint of $\delta: L_2(\M,\tau)\to L_2(\hat{\M},\tau)$. The infimum is taken over all $\xi\in \hat{\M}$ satisfying the continuity equation $z=\delta^*([\rho]\xi)$. Here the $L_2$-closure of $\delta^*(\A\delta(\A))$ is exactly $(I-E)L_2(\M)=L_2(\N)^\perp$, the orthogonal complement of $L_2(\N)$. So for $z\notin L_2(\N)^\perp$, $\norm{z}{g,\rho}=+\infty$.
Thus we only need to consider the metric $\norm{\cdot}{g,
\rho}$ restricted to
\[\mathcal{H}=\{a-E(a)\pl | \pl a=a^*\in \M\}\]
which is the horizantal direction on $S_B(\M)$.
Indeed, for any $z\in \mathcal{H}$ there exists a unique self-adjoint element $\xi\in \overline{\text{ran}(\delta)}=\ker(\delta^*)^\perp\in L_2(\hat{\M})$ such that
\begin{align} z=\delta^*([\rho]\xi )\pl, \pl \norm{z}{g,\rho}=\norm{\xi}{\rho}\pl. \label{5}\end{align}
(see \cite[Theorem 7.3]{CM} and \cite[Lemma 6.2]{CLSI}). Thus for an admissible smooth curve $\gamma:(a,b)\to S_B(\M)$, we have \[\gamma'(s)=\delta^*([\gamma(s)]D\gamma(s))\pl, \norm{\gamma'(s)}{g,\gamma(s)}=\norm{D\gamma(s)}{\gamma(s)}\pl.\]
The Wasserstein distance is then the (sub-)Riemannian distance induced by the metric $\lan \cdot,\cdot\ran_{g,\rho}$,
\[ W(\rho,\si)=\inf_\gamma \pl \int_0^1\norm{\gamma'(s)}{g,\gamma(s)}ds \]
where the infimum is taken over admissible smooth curve $\gamma\in C^1([0,1], S_B(\M)).$
In the following we denote by $\mathcal{H}_{\rho}$ the closure of $\mathcal{H}$ with respect to the $\norm{\cdot}{g,\rho}$ norm. $\mathcal{H}_{\rho}$ should be thought of as the horizantal tangent space at the point $\rho\in S_B(\M)$, equippied with sub-Riemannian metric $\norm{\cdot}{g,\rho}$. The element $z\in\mathcal{H}_{\rho}$ are  in one to one correspondence with $\xi\in \hat{\mathcal{H}}_{\rho}$ by the relation \eqref{5}.

Let $F:S_B(\M)\to \mathbb{C}$ be a function. We say $F$ admits a {\it (horizantal) gradient} at $\rho$ if there exists a vector $\xi\in \hat{\mathcal{H}}_\rho$ such that for every smooth path $\rho:(-\eps,\eps)\to S_B(\M)$ with $\rho(0)=\rho$,
\[  \rho'(0)=\delta^*([\rho]\xi_0)\quad \Longrightarrow \quad
  \frac{d}{dt}F(\rho(t))|_{t=0}
  \lel \langle \xi,\xi_0\rangle_{\rho} \pl ,\]
and we write $\xi=\grad_\rho F$.
By the relation \eqref{5}, this is equivalent to the gradient for the metric $\norm{\cdot}{g,\rho}$ in the usual Riemannian sense,
\[
  \frac{d}{dt}F(\rho(t))|_{t=0}
  \lel \langle \rho'(0) , \delta^*([\rho] \grad_\rho F) \rangle_{g,\rho} \pl .\]
An admissible smooth curve $\gamma:I\to S_B(\M)$ in the bounded density space is said to follow the path of \emph{steepest descent} or gradient flow with respect to $F$ if for any $a\in \A$ and $s\in (a,b)$
\[ \frac{d}{ds}\tau(a\gamma(s))=-\lan \delta(a),\grad_{\gamma(s)}F\ran_{\gamma(s)}\pl\pl,\]
or equivalently, $\gamma'(s)=-\delta^*([\gamma(s)] \grad_\gamma(s) F)$ weakly.
One immediate consequence is that  along a gradient flow $\gamma$, \begin{align}\label{steep} \frac{dF(\gamma(s))}{ds} \lel -\|\delta^*([\gamma(s)] \grad_\gamma(s) F)\|_{g,\gamma(s)}^2=-\|\grad_{\gamma(s)}F\|_{\gamma(s)}^2
  \pl. \end{align}
Now we take $F(\rho)=H(\rho)$ as the entropy functional. It is equivalent to take the relative entropy $D(\rho ||\N)$ because an admissible curve $E(\gamma(s))$ is independent of $s$ and $D(\gamma(s)||\N)=H(\gamma(s))-H(E(\gamma(s)))$. The next lemma shows that for $\rho\in S_B(\A_0)$,
$\rho_t=T_t(\rho)$ is the gradient flow of $H$ as well as other convenient properties of $\rho_t$. The key point is that it suffices to consider $\rho\in S_B(\A_0)$ for functional inequalities and we do not need assume curvature condition comparing to \cite{wirth},.

\begin{lemma}\label{conti} Let $\rho\in S_B(\A_0)$ and denote $\rho_t=T_t(\rho)$. Then
\begin{enumerate}\item[i)] $(\rho_t)$ is an admissible curve with $D(\rho_t)=\delta(\log \rho_t)$ and $\norm{D(\rho_t)}{\rho_t}=I(\rho_t)$.
\item[ii)] $t\mapsto I(\rho_t)$ is continuous and $(\rho_t)$ is the gradient flow with respect to entropy $H$.
\item[iii)] For any $t$, $W(\rho_t,\rho)<\infty$ and $\displaystyle \lim_{t\to \infty}W(\rho_t,\rho)=0$.
\item[iv)] $\displaystyle \lim_{t\to \infty}\norm{\rho_t-E(\rho)}{2}={0}$ and $\displaystyle \lim_{t\to \infty}D(\rho_t||\N)=0$.
\end{enumerate}
\end{lemma}

\begin{proof}
By assumption on $\A$, we have $T_t(\rho)\subset \A\cap \dom(A)$ and $\log\rho\in \dom(A^{1/2})$. Then we have the derivative
\[\frac{d}{dt}\rho_t= A\rho_t=\delta^*\delta(\rho_t)=\delta^*([\rho_t]\delta(\log\rho_t))\pl.\]
By definition \eqref{coneq}, this implies $D(\rho_t)=\delta(\log \rho_t)$.
\begin{align*}
\norm{\delta(\log \rho_t)}{\rho_t}^2=&\lan [\rho_t]\delta(\log\rho_t),  \delta(\log \rho_t) \ran\\=&\lan [\rho_t][\rho_t]^{-1}\delta(\rho_t), \delta(\log \rho_t) \ran
\\= &\tau( \delta(\rho_t)^*\delta(\log\rho_t))=\E(\rho_t,\log \rho_t)=I(\rho_t)
\end{align*}
where we have used the derivation relation $\delta(\log \rho)=J_{\log}^\rho(\delta(\rho))=[\rho]^{-1}\delta(\rho)$. The admissibility of $(\rho_t)$ follows from the continuity of $t\mapsto I(\rho_t)$. Indeed, by assumption $\mu_1 1\le \rho\le \mu_21$ and $A\rho\in L_2(\M)$. By the continuity of semigroup \cite[Proposition 3.1]{DL92}, we have $\rho_t\mapsto  \rho$ and $A\rho_t=T_t(A\rho)\mapsto  A\rho$ in $L_2$.
Since $f(x)=\log x$ is a Lipschitz continuous on $[\mu_1,\mu_2]$, $\displaystyle {\lim_{t\to 0}\norm{\log \rho_t-\log \rho }{2}= 0}$  by \cite[Corollary 7.5]{Double}. Then for the Fisher information,
\begin{align*}\lim_{t\to 0}I(\rho_t)-I(\rho)=&\lim_{t\to 0} \tau( A\rho_t \log \rho_t)-\tau( A\rho \log \rho)\\ \le & \lim_{t\to 0}\tau( A\rho_t (\log \rho_t-\log \rho)) +\tau( (A\rho_t-A\rho) \log \rho)
\\ =& \lim_{t\to 0} \tau( T_t(A \rho)(\log \rho_t-\log \rho)) +\tau( (T_t(A\rho)-A\rho) \log \rho)
\\ \le& \lim_{t\to 0} \norm{T_t(A \rho)}{2}\norm{\log \rho_t-\log \rho}{2} +\tau( \norm{T_t(A \rho)-A\rho}{2}\norm{\log \rho}{2} =0\pl.
 \end{align*}
Applying semigroup property, we have $t\mapsto I(\rho_t)$ is continuous. For the gradient flow,
given a self-adjoint $\beta=\delta^*([\rho]\xi_0)$ ,
\begin{align*}
\frac{d}{dt}H(\rho+t\beta)|_{t= 0}&\lel \tau(\beta \log \rho)=\lan \delta^*([\rho]\xi_0),  \log \rho\ran_\tau
\\=&\lan [\rho]\xi_0,  \delta(\log \rho)\ran_\tau
=\lan \xi_0,  \delta(\log \rho)\ran_\rho
 \pl .
 \end{align*}
Thus $\grad_\rho H=\delta(\log\rho)$ and the gradient flow for $H(\cdot)$ is given by the equation
\begin{align*} \rho'(t)&=
-\delta^*([\rho(t)]\grad_{\rho(t)}H)
 \lel -\delta^*([\rho(t)]\delta(\log \rho(t)))
 \\ \lel& -\delta^*([\rho(t)][\rho(t)]^{-1}\delta(\rho(t)))\lel -A(\rho(t)) \pl ,\end{align*}
 whose solution is the semigroup flow $\rho(t)=T_t(\rho(0))$. For iii), since $s\mapsto\rho_s$ is admissible
 \begin{align*}&\lim_{t\to 0}W(\rho_t,\rho)\le \lim_{t\to 0}\int_{0}^t \norm{D\rho_s}{\rho_s}ds=\lim_{t\to 0}\int_{0}^t I(\rho_s)^{1/2}ds=0\pl.\end{align*}
For iv), we first show the $L_2$-convergence. Consider $A$ as a positive self-adjoint operator on $L_2(\M)$ and denote $e_s$ (resp. $e_0$) as the spectral projection for the spectrum $[0,s)$ (resp. $\{0\}$). Clearly, $e_0(L_2(\M))=L_2(\N)$. Write $\o{\rho}=\rho-E(\rho)$. We have \[\o{\rho}\in e_0^\perp\pl, \pl\norm{\o{\rho}}{2}\le \norm{\rho}{2}\pl, \pl T_t(\rho)-E(\rho)= T_t(\o{\rho})\pl.\] Then $\lim_{s\to 0}\norm{e_s(\o{\rho})}{2}=0$. For any $\epsilon>0$, we can find $s>0$ and then large enough $t$ such $\norm{e_s(\o{\rho})}{2}<\epsilon$ and $e^{-st}\norm{\rho}{2}<\epsilon$. Thus
\begin{align*}
\norm{T_t(\rho)-E(\rho)}{2}=&\norm{T_t(\o{\rho})}{2}\le \norm{T_t(e_s(\o{\rho}))}{2}+\norm{T_t(\o{\rho}-e_s(\o{\rho}))}{2}
\\ \le& \epsilon +e^{-st}\norm{\rho}{2}\le 2\epsilon\pl.
\end{align*}
Therefore $\lim_{t\to \infty}\norm{T_t(\rho)-E(\rho)}{2}=0$. This further implies $\lim_{t\to \infty}\norm{T_t(\rho)-E(\rho)}{1}=0$ and by Lemma \ref{continuity},
\begin{align*}&\lim_{t\to \infty}D(T_t(\rho)||\N)=D(E(\rho)||\N)=0\pl. \qedhere \end{align*}
\end{proof}
\section{Fisher monotonicity and CB-return time}
\subsection{Monotonicity of Fisher Information}

Our first ingredient is the monotonicity of Fisher information, which can be equivalently characterized by the following conditions.
\begin{prop}\label{monotone}
Let $\la\in \mathbb{R}$. For a state $\rho\in S(\M)$, denote $T_t(\rho)=\rho_t$.
The following conditions are equivalent
\begin{enumerate}
\item[i)] for any $\rho\in S_B(\A_0)$ and $t\ge 0$,
\[I(\rho_t)\le e^{-2\la t}I(\rho) .\]
\item[ii)] for any $\rho\in S_H(\M)$ and $s,t\ge0$,
\[D(\rho_t||\N)-D(\rho_{s+t}||\N)\le e^{-2\la t}(D(\rho||\N)-D(\rho_s||\N))\pl.\]
\item[iii)] for any $\rho\in S_H(\M)$ and $s,t\ge0$,
\[H(\rho_t)-H(\rho_{s+t})\le e^{-2\la t}(H(\rho)-H(\rho_{s}))\pl.\]
\end{enumerate}
\end{prop}
\begin{proof}
Let $\rho\in S_B(\A_0)$. Combined Lemma \eqref{conti} with \cite[Proposition 5.23]{wirth}), we have for $\rho\in S_B(\A_0)$
\begin{align}&D(\rho||\N)-D(\rho_t||\N)=H(\rho)-H(\rho_t)=\int_0^t I(\rho_u)du, \nonumber\\
&I(\rho)=\lim_{t\to 0} \frac{D(\rho||\N)-D(\rho_t||\N)}{t}. \label{derivative}
\end{align}
Then ii) follows from i) since for $\rho\in S_B(\A_0)$,
\begin{align*}
D(\rho||\N)-D(\rho_s||\N)&=\int_0^s I(\rho_u)du \ge \int_0^s e^{2\la t}I(\rho_{t+u})du
\\& = e^{2\la t}\big(\int_t^{s+t} I(\rho_u)du\big)=e^{2\la t}\big(D(\rho_t||\N)-D(\rho_{s+t}||\N)\big)\pl.
\end{align*}
For general $\rho\in S_H(\M)$, we use the approximation in Lemma \ref{equiv}.
On the other hand, i) follows from ii) since for $\rho\in S_B(\A_0)$,
\begin{align*}
I(\rho)&=\lim_{s\to 0}\frac{D(\rho||\N)-D(\rho_s||\N)}{s}\\ &\ge \lim_{s\to 0}e^{2\la t}\frac{D(\rho_t||\N)-D(\rho_{s+t}||\N)}{s} \ge e^{2\la t}I(\rho_t)\pl.
\end{align*}
The equivalence to iii) follows from the fact that $D(\rho||\N)=H(\rho)-H(E(\rho))$ for $\rho\in S_H(\M)$ and $E(\rho)=E(T_t(\rho))$. \end{proof}

\begin{defi}We say a semigroup $T_t$ is $\la$-Fisher monotone for $\la\in \mathbb{R}$ (in short, $\la$-FM) if $T_t$ satisfies one of the above conditions in Proposition \eqref{monotone}. We say $T_t$ is $\la$-complete Fisher monotone ($\la$-CFM) if for any finite von Neumann algebra $\R$, $id_\mathcal{R}\ten T_t$ is $\la$-FM. For $\la=0$, we simply say $T_t$ is (complete) Fisher monotone.
\end{defi}
The idea of following proposition goes back to the $\Gamma$-calculus in \cite{BE85}.
\begin{prop}\label{FM} For $\la >0$, $\la$-FM implies $\la$-MLSI.
\end{prop}
\begin{proof}For $\rho\in S_B(\A_0)$, denote $f(t)=D(\rho||\N)-D(\rho_t||\N)$ and hence $I(\rho_t)=f'(t)$.
Then $\la$-FM means that
\[ f'(t)\le e^{-2\la t}f'(0) \]
Integrating both sides from $0$ to $t$,
\[ D(\rho||\N)-D(\rho_t||\N)\le \frac{e^{-2\la t}-1}{-2\la} I(\rho) \]
Taking $t\to \infty$,
\[ 2\la D(\rho||\N)=\lim_{t\to \infty }2\la (D(\rho||\N)-D(\rho_t||\N)) \le \lim_{t\to \infty } (1-e^{-2\la t}) I(\rho)=I(\rho)\pl,\]
which this is $\la$-MLSI. Here we used the assumption $\la>0$ and the property $\displaystyle{\lim_{t\to \infty}D(\rho_t||\N)= 0}$ from \ref{conti}.
\end{proof}

\subsection{Complete bounded return time}
Let $\M$ be a finite von Neumann algebra and $\N\subset\M$ be a subalgebra. The conditional $L_1$ space $L_\infty^1(\N\subset\M)$ is defined as the completion of $\M$ with respect to the norm
\[ \norm{x}{L_\infty^1(\N\subset\M)}=\sup_{a,b\in L_2(\N)\pl,\|a\|_2=\|b\|_2=1}\norm{axb}{1}\pl,  \]
where the supremum takes over all $a,b\in L_2(\N)$ with $\|a\|_2=\|b\|_2=1$.
The operator space structure of $L_\infty^1(\N\subset\M)$ is given by
\[M_n(L_\infty^1(\N\subset\M))=L_\infty^1(M_n(\N)\subset M_n(\M))\pl.\]
(see \cite{JPMemo} and \cite[Appendix]{GJL19}).
We consider again $T_t:\M\to \M$ be a symmetric quantum Markov semigroup and $\N$ be the fixed point subalgebra with conditional expectation $E$. We define the complete bounded (CB) return time
 of $T_t$ as follows
\[t_{cb}=\inf \{ \pl t\ge 0 \pl |\pl \norm{T_t-E:L_\infty^1(\N\subset\M)\to L_\infty(\M)}{cb}\le 1/2\}\]
If such $t$ does not exist, we write  $t_{cb}=+\infty$. Recall the following lemma from \cite{CLSI}.
\begin{lemma}[Lemma 3.15 of \cite{CLSI}] Let $T:M\to M$ be a unital completely positive $\N$-bimodule map such that
  \[ \|T-E:L_\infty^1(N\subset M)\to M\|_{cb} \kl \frac12 \pl .\]
Then $T\ge_{cp} \frac{1}{2}E$, i.e. $T-\frac{1}{2}E$ is completely positive.
\end{lemma}
We refer \cite{CLSI} for the complete proof and illustrate here the argument for the ergodic case. Namely, we consider $\N=\mathbb{C} 1$ and $L_\infty^1(\N\subset \M)=L_1(\M)$.
The CB return time becomes
\[t_{cb}=\inf \{\pl t\ge 0 \pl | \pl \norm{T_t-E:L_1(\M)\to L_\infty(\M)}{cb}\le 1/2\}\pl.\]
This completely bounded norm is by no means abstract. Indeed, by Effros-Ruan Theorem (see \cite{ER00} and also \cite{BP}),
\[\norm{T:L_1(\M)\to L_\infty(\M)}{cb}=\norm{C_T}{\M^{op}\overline{\ten}\M}\]
where $C_T$ is the kernel of $T$ (also called Choi matrix, in finite dimensions) given by the relation
\[ T(a)=\tau \ten \id (C_T(a\ten 1))\pl, a\in L_1(\M)\cong (\M^{op})_*\pl.\]
Here $\M^{op}$ is the opposite algebra of $\M$. Moreover, the correspondence $T\leftrightarrow C_T$ is also order preserving: $T$ is completely positive if and only if $C_T$ as a operator is positive in $\M^{op}\overline{\ten }\M$. In particular, the conditional expectation onto scalars $E_\tau(a)=\tau(a)1$ has kernel as the identity $1\ten 1\in \M^{op}\overline{\ten} \M$. For this special case,
\begin{align*} & \norm{T-E_\tau:L_1(\M)\to \M}{cb}\le 1/2 \Longleftrightarrow \pl \norm{C_T-1\ten 1}{\M^{op}\overline{\ten}\M} \le 1/2 \\ \Longrightarrow & \pl C_T\le \frac{1}{2}1\ten 1 \pl \Longleftrightarrow \pl C_T\ge_{cp} 1/2C_{E_\tau}\pl,\end{align*}
where the implication ``$\Rightarrow$'' is evident from spectrum calculus for a self-adjoint operator $C_T$. This proves the above lemma for the special case $\N=\mathbb{C}1$. The general case for non-trivial $\N$ is an extension for bimodule maps.

The next lemma shows $t_{cb}$ is the half-life for the decay of relative entropy.
\begin{lemma}\label{decay}
Let $\N\subset\M$ be a subalgebra and $E$ be the condition expectation onto $\N$.
Suppose for $\al\in (0,1)$, $\Phi-\al E$ is a positive map
and $\Phi(L_1(\N))\subset L_1(\N)$.
 Then for any $\rho\in S(\M)$,
\begin{align}\label{7}D(\Phi(\rho)||\N)\le (1-\al) D(\rho||\N)\pl.\end{align}
If in additional, $\Phi-\al E$ is a completely positive map, the same assertion holds for $\Phi \ten \id_\mathcal{R}$.
\end{lemma}
\begin{proof}
Define $\Psi:=\frac{1}{1-\al}(\Phi-\al E)$. By assumption that $\Phi-\al E$ is positive , $\Psi$ is a positive trace preserving map such that $\Psi(L_1(\N))\subset L_1(\N)$. Thus $\Phi=(1-\al) \Psi+\al E$. Note that
the data processing inequality holds for positive trace preserving maps  \cite{Jenvcova1}. Then by the convexity of relative entropy and the data processing inequality of $D(\cdot||\N)$ give
\begin{align*}D(\Phi(\rho)||\N)
= &D((1-\al) \Psi(\rho)+\al E(\rho)||\N)
\le (1-\al) D(\Psi(\rho)||\N)+\al D(E(\rho)||\N)
\\= &(1-\al) D(\Psi(\rho)||\N)\le (1-\al) D(\rho||\N)\pl.
\end{align*}
The same argument applies to $\Phi\ten \id_\mathcal{R}$.
\end{proof}
We now prove our main technical theorem that (complete) Fisher monotonicity plus CB-return time implies MLSI (resp. CLSI).
Define the function \[\kappa(\la,t)=\begin{cases}
                 \frac{1}{4t}, & \mbox{if } \la=0 \\
                 \frac{\la}{2(1-e^{-2\la t})}, & \mbox{if } \la\neq 0.
               \end{cases}.\] For each $t$, $\la\mapsto \kappa(\la,t)$ is continuous at $0$.
\begin{theorem}\label{CLSI1}
Let $T_t:\M\to \M$ be a symmetric quantum Markov semigroup. Suppose \begin{enumerate}\item[i)]$T_t$ satisfies $\la$-FM for some $\lambda\in \mathbb{R}$
\item[ii)] $T_t$ has finite CB-return time $t_{cb} < \infty$.
\end{enumerate}
Then $T_t$-satisfies $\kappa(\la,t_{cb})$-MSLI. The same assertions holds replacing ``FM'' with ``CFM'' and ``MLSI'' with ``CLSI''.
\end{theorem}
\begin{proof} Write $t_{cb}=t_0$.
 As a consequence of Lemma \ref{decay}, we have
\[D(T_{t_0}(\rho)||\N)\le \frac{1}{2} D(\rho||\N)\]
Let $n>1$ be an integer and write $T_t(\rho)=\rho_t$.
For $\la< 0$, we have
\begin{align*}\frac{1}{2}D(\rho||\N)&\le D(\rho||\N)-D(\rho_{t_0}||\N)
 =\sum_{j=0}^{n-1}D(\rho_{\frac{jt_0}{n}}||\N)-D(\rho_{\frac{(j+1)t_0}{n}}||\N)
\\ &\le \sum_{j=0}^{n-1}e^{-2\la \frac{jt_0}{n} }(D(\rho||\N)-D(\rho_{\frac{jt_0}{n}}||\N))
\\ &= \frac{1-e^{-2\la t_0}}{1-e^{-2\la \frac{t_0}{n} }} (D(\rho||\N)-D(\rho_{\frac{t_0}{n}}||\N))
\end{align*}
where we used $\la$-FM in the second inequality. Rearranging the terms, we have
\[D(\rho_{\frac{t_0}{n}}||\N)\le \frac{-e^{-2\la t_0}+\frac{1}{2}+\frac{1}{2}e^{-2\la \frac{t_0}{n}}}{1-e^{-2\la t_0}}D(\rho||\N)\]
For $\rho\in S_B(\A_0)$, $t\mapsto D(\rho_t||\N)$ is differentiable and $\frac{d}{dt}D(\rho_t||\N)|_{t=0}=-I(\rho)$. Taking the limit $n\to \infty$, we have
\begin{align*}
I(\rho)&=\lim_{n\to \infty}\frac{D(\rho||\N)-D(\rho_{\frac{t_0}{n}}||\N)}{\frac{t_0}{n}}
\\&\ge
\lim_{n\to \infty}\frac{n}{t_0}(1-\frac{-e^{-2\la t_0}+\frac{1}{2}+\frac{1}{2}e^{-2\la \frac{t_0}{n}}}{1-e^{-2\la t_0}})
\\&=
\lim_{n\to \infty}\frac{n}{t_0}(\frac{\frac{1}{2}-\frac{1}{2}e^{-2\la \frac{t_0}{n}}}{1-e^{-2\la t_0}})
=\frac{\la}{{1-e^{-2\la t_0}}}D(\rho||\N)
\end{align*}
which is $\frac{\la}{2(1-e^{-2\la t_0})}$-MLSI.
The argument above remains valid for $\la=0$ and $T_t\ten \id_{\mathcal{\R}}$. This completes the proof.
\end{proof}
\begin{rem}{\rm For the ergodic classical Markov semigroups, it was proved by Diaconis and Saloff-Coste in \cite[Theorem 3.10]{DSC} that the bound return time (the complete boundness is automatic here)
\[t_\infty:=\{ t\ge 0 \pl |\pl \norm{T_t-E:L_1(\Omega)\to L_\infty(\Omega)}{}\le 1  \}\] itself implies $\frac{1}{t_\infty}$-LSI, which further implies MLSI. Nevertheless, their argument went through hypercontractive estimate that does not apply to non-commutative non-ergodic setting.}
\end{rem}

The CB-return time can be estimated by standard argument.
\begin{prop}\label{esti}Let $T_t=e^{-At}:\M\to \M$ be a symmetric quantum Markov semigroup and $\N$ be its fixed-point subalgebra. Suppose
\begin{enumerate}
\item[i)] for some $t_0\ge 0$,
$\|T_{t_0}:L^1_\infty(\N\subset \M)\to L_{\infty}(\M)\|_{cb}\kl C$.
\item[ii)] the generator $A$ has spectral gap $\si>0$, that is
 \[ \|A^{-1}(\id-E):L_2(\M)\to L_2(\M) \|_{}\kl \si^{-1} \pl .\]
\end{enumerate}
Then for $t\ge t_0$,
\[\norm{T_t-E:L_\infty^1(\N\subset \M)\to L_\infty(\M)}{cb}\le Ce^{-\si (t-t_0)}\pl.\]
As a consequence, $t_{cb}\le \si^{-1} (\log 2C)+t_0$.
\end{prop}
\begin{proof} Note that $T_t-E$ is an $\N$-bimodule map. We have for $t\ge t_0$,
\begin{align*}&\norm{T_t-E:L_\infty^1(\N\subset \M)\to L_\infty(\M)}{cb}\\=&\norm{T_{\frac{t}{2}}-E:L_\infty^1(\N\subset \M)\to L_\infty^2(\N\subset \M)}{cb}^2
\\=&\norm{T_{\frac{t_0}{2}}:L_\infty^1(\N\subset \M)\to L_\infty^2(\N\subset \M)}{cb}^2\norm{T_{\frac{t-t_0}{2}}-E:L_\infty^2(\N\subset \M)\to L_\infty^2(\N\subset \M)}{cb}^2
\\=&\norm{T_{t_0}:L_\infty^1(\N\subset \M)\to L_\infty(\M)}{cb}\norm{T_{\frac{t-t_0}{2}}-E:L_2(\M)\to L_2(\M)}{cb}^2
\\=&Ce^{-\si(t-t_0)}\pl,
\end{align*}
Here the first equality uses \cite[Lemma 3.13]{CLSI} and the third equality uses \cite[Lemma 3.12]{CLSI}.
\end{proof}

The above estimates has the following two corollaries. The first one is the non-ergodic version of \cite[Proposition 3.2]{CLSI}. It basically says that the spectral gap plus a non-ergodic Varopoulos dimension condition implies finite CB-return time.
\begin{lemma}\label{sfc} Let $T_t:\M\to \M$ be a symmetric quantum Markov semigroup and $\N\subset \M$ be the fixed-point subalgebra. Suppose
 \begin{enumerate}
 \item[i)] $\|T_t:L^1_\infty(\N\subset \M)\to L_{\infty}(\M)\|_{cb}\kl c t^{-d/2}$ for some $c,d>0$ and all $0< t< 1$;
 \item[ii)] the generator $A$ has spectral gap $\si>0$
 \end{enumerate}
Then the CB-return time satisfies
 \[t_{cb}\le \frac{1}{2}+\frac{d-1}{2}
 \log 2+\frac{1}{\si}\log c\]
 \end{lemma}
 \begin{proof}Choose $t_0=1/2$ in Lemma \ref{esti}.
 \end{proof}

The second cases is related to finite von Neumann subalgebra index. Recall that for two states $\rho,\omega$, the maximal relative entropy is \[D_\infty(\rho||\omega)=\log \inf\{\pl \al>0 \pl   | \pl \rho\le \al \omega \pl \}\pl.\]
For an inclusion $\N\subset \M$ of finite von Neumann algebras, the maximal relative entropy $D_\infty$ of $\M$ to $\N$ and its CB-version $D_{\infty,cb}$ is defined as
\[D_\infty(\M||\N)=\sup_{\rho \in S(\M)}D_\infty(\rho||\N)\pl, \pl D_{\infty,cb}(\M||\N)=\sup_{m}D_\infty(M_m(\M)||M_m(\N))\]
It was proved in \cite[Theorem 3.9]{GJL19} that
\begin{align*}
D_{\infty,cb}(\M||\N)=&\log \norm{\operatorname{id}:L_\infty^1(\N\subset\M)\to L_\infty(\M)}{cb}\pl.
\end{align*}
The next proposition gives the estimate of $t_{cb}$ given that $D_{\infty,cb}(\M||\N)$ is finite and spectral gap is positive.

\begin{prop}\label{fd}Let $T_t:\M\to \M$ be a symmetric quantum Markov semigroup and $\N$ be its fixed-point subalgebra. Suppose $D_{cb,\infty}(\M||\N)<\infty$ is finite and $T_t$ has spectral gap $\si>0$. Then
\[\norm{T_t-E:L_\infty^1(\N\subset \M)\to L_\infty(\M)}{cb}\le e^{D_{cb,\infty}(\M||\N)}e^{-\si t}\pl.\]
As a consequence, $t_{cb}\le \si^{-1} (D_{cb,\infty}(\M||\N)+\log 2)$.
\end{prop}
\begin{proof} Choose $t_0=0$ in Lemma \ref{esti}.
\end{proof}

The the maximal relative entropy $D_{cb,\infty}(\M||\N)$ connects to the von Neumann algebra subalgebra index and is explicit for many examples. It was proved in \cite[Theorem 3.1]{GJL19} that $D_{\infty}(\M||\N)=\log \la(\M:\N)^{-1}$ for $\M,\N$ being II$_1$ factors or finite dimensional, where $\la(\M:\N)$ is the Pimsner-Popa index in \cite{pipo}. In particular, for II$_1$ factors,
$D_{cb,\infty}(\M||\N)=\log [\M:\N]$ where $[\M:\N]$ is the Jones subfactor index; for $\M,\N$ finite dimensional, the explicit formula of $D_\infty(\M||\N)$
is calculated in \cite[Theorem 6.1]{pipo}, from which $D_{cb,\infty}(\M||\N)$ are also known.
For example, \[D_{cb,\infty}(M_n||\C)=2\log n \pl , \pl D_{cb,\infty}(M_n||l_\infty^n)=\log n \pl, \pl  D_{cb,\infty}(l_\infty^n||\C)=\log n.\] For any $\N\subset M_n$, $D_{cb,\infty}(M_n||\N)\le D_{cb,\infty}(M_n||\C 1)=2\log n$.

\subsection{Entropy Ricci curvature bound}
We shall now discuss the connection between Fisher monotonicity and Ricci curvature lower bound and give a non-egordic version of Bakry-Emery theorem. Following \cite{EF}, we call Ricci curvature bound defined through geodesic convexity of $D$ as \emph{entropy Ricci curvature bound}. We first review the different formulations of entropy Ricci curvature bound discussed in \cite{wirth,CM18,DR}.
For a function $f:[0,a)\to \infty$, we introduce the notation
\[\frac{d^+}{dt}f=\limsup_{t\to 0} \frac{1}{t}(f(t)-f(0))\pl. \]
Recall that $S_H(\rho)=\{ \rho \in S(\M)\pl |\pl H(\rho)<\infty\}$ is the state space with finite entropy and we write $\rho_t=T_t(\rho)$

\begin{defi}\label{definition}Let $T_t=e^{-At}$ be a symmetric quantum Markov semigroup and let $(\A,\delta,\hat{\M})$ be a derivation triple of $T_t$. For $\la\in \mathbb{R}$, define the following conditions
\begin{enumerate}
\item[i)]Gradient Estimate: we say $T_t$ satisfies a $\la$-gradient estimate ($\la$-GE) if for any $\rho\in S(\M)$ and $x\in \dom(A^{1/2})$ with $ E(x)=0$,
\[\norm{\delta(T_t(x))}{\rho}^2\le e^{-2\la t} \norm{\delta(x)}{\rho_t}^2\pl, \forall t\ge 0\pl.\]
\item[ii)]Evolution Variational Inequality:  we say $T_t$ satisfies a $\la$-evolution variational inequality ($\la$-EVI) if for all $\rho,\si\in S_H(\M)$ with $W(\rho,\si)<\infty$ and $t\ge 0$
\[\frac{1}{2}\frac{d^+}{dt}W(\rho_t,\si)^2+\frac{\la}{2}W(\rho_t,\si)^2+H(\rho_t)\le H(\si) \pl.\]
\item[iii)]Displacement Convexity: we say the entropy functional $H$ is geodesically $\la$-convex if for any constant speed geodesic $\gamma:[0,1]\to (S_H(\M),W)$,
    \[ H(\gamma(s))\le (1-s)H(\gamma(0))+sH(\gamma(1))-\frac{\la(1-s)s}{2}W(\gamma(0),\gamma(1))^2\pl.\]\end{enumerate}
\end{defi}
When $\M$ is a finite dimensional $C^*$-algebra and $T_t$ being a primitive semigroup (including non-symmetric cases), all three of the above conditions are proved to be equivalent and are referred to as a {\it $\la$-Ricci lower bound} in \cite{DR,CM18}. For finite von Neumann algebras $\M$, it has been proved in \cite[Theorem 7.12]{wirth} that
\begin{align*}\text{(i)} \Rightarrow \text{$W$ is non-degenerate and (ii)} \Rightarrow \text{$(S_H(\M),W)$ is a geodesic space and (iii)} \end{align*}
For this reason, we take the gradient estimate condition $\la$-GE as our working definition of entropy Ricci curvature bound.

\begin{rem}{\rm \label{DH}For EVI and displacement convexity above, it is equivalent to replace the entropy $H(\rho)$ by the relative entropy $D(\rho||\N)$. This is because for $\rho\in S_H(\M)$, $D(\rho||E(\rho))= H(\rho)-H(E(\rho))<\infty$.
For $\la$-EVI, $W(\rho,\si)<\infty$ implies $E(\rho)=E(\si)$ and hence
\begin{align}\label{6}
\frac{1}{2}\frac{d^+}{dt}W(\rho_t,\si)^2+\frac{\la}{2}W(\rho_t,\si)^2+D(\rho_t||\N)\le D(\si||\N) \pl.
\end{align}
For $\la$-displacement convexity, $E(\gamma(s))=E(\gamma(t))$ for any admissible curve $\gamma$ and hence
\begin{align*}
&D(\gamma(s)||\N)\le (1-s)D(\gamma(0)||\N)+sD(\gamma(1)||\N)-\frac{\la(1-s)s}{2}W(\gamma(0),\gamma(1))^2\pl.
\end{align*}
}
\end{rem}


\begin{rem}{\rm
A semigroup $T_t$ can admit distinct derivation triples $(\A,\delta,\hat{\M})$. For example, let $M_2$ be $2\times 2$ matrix algebra and consider the depolarizing semigroup \[D_t:M_2\to M_2, D_t(\rho)=e^{-t}\rho+ (1-e^{-t})\tau(\rho) 1\pl,\] where $\tau$ is the normalized trace $\tau(\rho)=\frac{1}{2}\text{Tr}(\rho)$. It was discussed in \cite[Section 5.6]{CM18} that $D_t$ admits a derivation
\[\delta: M_2\to \oplus_{j=1}^3 M_2\pl, \delta(a)=\frac{1}{2\sqrt{2}}(i[X,a], i[Y,a],  i[Z,a])\pl.\]
where $X,Y,Z$ are Pauli matrices. This follows from that the depolarizing map $E$ is an average of unitary conjugation by Pauli matrices,  \[E(\rho)=\frac{tr(\rho)}{2} 1 =\frac{1}{4}(\rho+X\rho X+Y\rho Y+Z\rho Z)\pl.\]
On the other hand, the depolarizing map $E$ can also be seen as the following average of unitary conjugations over the unitary group $U(2)\subset M_2$,
\[E(\rho)=\int_{U(2)} u^*\rho u\pl  d\mu(u)\]
where $\mu$ is the Haar measure on $U(2)$. Then one can construct an alternative derivation
\[ \tilde{\delta}:M_2\to L_\infty(U(2), M_2)\pl, \tilde{\delta}(a)(u)=i[u,a]\pl,\]
where $L_\infty(U(2), M_2)$ is the $M_2$-valued function on the Lie group $U(2)$. For more examples of distinct derivation triple, see Example \ref{finitegroup}.}
\end{rem}
The next proposition shows that the gradient estimate is independent of the choice of derivation triple $(\A,\hat{\M},\delta)$.
\begin{prop}The definition of the gradient estimate is independent of the choice of derivation.
\end{prop}
\begin{proof}We show that the norm
\[\norm{\delta(x)}{L_2(\hat{\M},\rho)}^2=\int_{0}^{1}\tau(\delta(x)^*\rho^{s}\delta(x)\rho^{1-s}) ds\]is independent of $\delta$. Recall that the Dirichlet algebra $\A_\E=\dom(A^{1/2})\cap \M$ is a core for $\delta$ and closed under $C^1$-functional calculus. For $x,\rho\in \A_\E$, we have $\rho^s\in \A_\E$ and by Leibniz rule
\[\rho^{s}\delta(x)=\delta(x\rho^{s})-\delta(\rho^{s})x\pl, x\in \dom(A^{1/2})\pl.\]
Then for each $s\in [0,1]$,
\begin{align*}
\tau(\delta(x)^*\rho^{s}\delta(x)\rho^{1-s}) =&\tau(\delta(x)^*\delta(\rho^{s}x)\rho^{1-s}) - \tau(\delta(x)^*\delta(\rho^{s})x\rho^{1-s})
\\ =& \tau\Big(E_M(\delta(x)^*\delta(\rho^{s}x)\rho^{1-s})\Big)- \tau\Big(E_M(\delta(x)^*\delta(\rho^{s})x\rho^{1-s})\Big)
\\ =& \tau\Big(E_M(\delta(x)^*\delta(\rho^{s}x))\rho^{1-s}\Big) - \tau\Big(E_M(\delta(x)^*\delta(\rho^{s}))x\rho^{1-s}\Big)
\\ =& \tau\Big(\Gamma(x,\rho^{s}x)\rho^{1-s}\Big) - \tau\Big(\Gamma(x,\rho^{s})x\rho^{1-s}\Big),
\end{align*}
which is completely determined by gradient form $\Gamma$. We now show for general $\rho,x$, $\{\norm{\delta(x)}{L_2(\hat{\M},\rho)}\}$ can be approximated by $\rho,x\in\A_\E$. For $x\in \dom(A^{1/2})$, we chose a sequence $x_n\to x$ in the graph norm of $\delta$. In particular, $\delta(x_n)\to \delta(x)$ in $L_2$. Then for $\rho\in\A_\E$,
\begin{align*}
\lim_{n\to \infty}\tau(\delta(x_n)^*\rho^{s}\delta(x_n)\rho^{1-s})
 =& \lim_{n\to \infty}\norm{\rho^{s/2}\delta(x_n)\rho^{(1-s)/2}}{2}^2
 \\= &\norm{\rho^{s/2}\delta(x)\rho^{(1-s)/2}}{2}^2
=\tau(\delta(x)^*\rho^{s}\delta(x)\rho^{1-s})
\end{align*}
For any $\rho\in S(\M)$, we take sequence $\rho_n=\rho\wedge n\in \A_\E$ and $\rho_n\nearrow \rho$ in $L_1$. Then for any $x\in \dom(A^{1/2})$, we apply the Fatou lemma
\begin{align*}
\limsup_{n\to \infty}\tau(\delta(x)^*\rho_n^{s}\delta(x)\rho_n^{1-s})
\le&\limsup_{n\to \infty}\tau(\delta(x)^*\rho_n^{s}\delta(x)\rho^{1-s})
\\\le &\tau(\delta(x)^*\rho^{s}\delta(x)\rho^{1-s})
\le \liminf_{n\to \infty}\tau(\delta(x)^*\rho_n^{s}\delta(x)\rho_n^{1-s})
\end{align*}
which implies $\tau(\delta(x)^*\rho_n^{s}\delta(x)\rho_n^{1-s})\nearrow \tau(\delta(x)^*\rho^{s}\delta(x)\rho^{1-s})$. Then by monotone convergence theorem, $\lim_{n}\norm{\delta(x)}{L_2(\hat{\M},\rho_n)}^2=\norm{\delta(x)}{L_2(\hat{\M},\rho)}^2$. That completes the proof.
\end{proof}
The next proposition shows that entropy Ricci curvature bound implies Fisher monotonicity.
\begin{prop}\label{fisher}For any symmetric quantum Markov semigroup $T_t: \M\to\M$ and $\la\in \mathbb{R}$, $\la$-GE implies $\la$-FM.
\end{prop}
\begin{proof}Let $\rho\in S_B(\A_0)$ and $\rho_t=T_t(\rho)$ be the semigroup path. By
Lemma \ref{conti}, $(\rho_t)$ is an admissible curve with
\[\norm{D\rho_t}{}^2=\norm{\delta(\log\rho_t)}{\rho_t}^2=I(\rho_t)\pl,\]
and $t\mapsto I(\rho_t)$ is continuous. Then it follows from \cite[Theorem 6.9]{wirth} that for any $s>0$
\begin{align*}& I(T_{s+t}(\rho))=I(T_s(\rho_t))=\norm{D T_s(\rho_t)}{T_s(\rho_t)}^2\le e^{-2\la s}\norm{D \rho_t}{\rho_t}^2=e^{-2\la s}I(\rho_t)\pl. \qedhere \end{align*}
\end{proof}

For $\la>0$, the above Proposition and  Proposition \ref{FM} combined gives $\la$-GE $\Rightarrow$ $\la$-FM $\Rightarrow$ $\la$-MLSI, which is a noncommutative non-ergodic version of Bakry-Emery theorem. In the following, we take an another approach using Otto-Villani's HWI inequality introduced in \cite{OV}. The quantum HWI inequality is obtained in \cite[Corollary 2]{DR} for finite dimensional ergodic case (see also \cite{CM18}). For finite von Neumann algebra, this idea is also used in \cite[Proposition 7.9]{wirth}. Here the major difference to \cite{wirth} is that we do not need to assume $\la$-GE for some $\la>0$.
\begin{theorem}\label{HWI}
Let $T_t$ be a semigroup satisfying $\la$-EVI for $\la\in \mathbb{R}$: for any $\rho,\si\in S_H(\M)$ with $W(\rho,\si)<\infty$,
\[\frac{1}{2}\frac{d^+}{dt}W(\rho_t,\si)^2+\frac{\la}{2}W(\rho_t,\si)^2+H(\rho_t)\le H(\si) \pl.\]
 Then $T_t$ satisfies the following $\la$-HWI inequality: for any $\rho\in S_B(\A_0),\si\in S_H(\M)$ with $W(\rho,\si) <
\infty$,
\[H(\rho)-H(\si)\le W(\rho,\si)\sqrt{I(\rho)}-\frac{\la}{2}W(\rho,\si)^2 \pl, \]
\end{theorem}
\begin{proof}By Lemma \ref{conti}, we know that for $\rho\in S_B(\A_0)$, $t\mapsto I(\rho_t)$ is continuous and $t\mapsto\rho_t$ is an admissible curve with $\norm{D\rho_t}{\rho_t}^2= I(\rho_t)$. By triangle inequality,
\begin{align*}\frac{d}{dt}^+ W(\rho_{t+s},\si)\le& \limsup_{t\to 0}\frac{1}{t}W(\rho_{t+s},\si)-W(\rho_s,\si)
\\ \le & \limsup_{t\to 0}\frac{1}{t}W(\rho_{t+s},\rho_s) \le \limsup_{t\to 0}\frac{1}{t}\int_{0}^t \norm{D\rho_{t+s}}{\rho_{t+s}}ds= \sqrt{I(\rho_s)}\pl,\end{align*}
Therefore,
\begin{align*} -\frac{1}{2}\frac{d^+}{dt}W(\rho_t,\si)^2&=\liminf_{t\to 0}\frac{1}{2t}(W(\rho,\si)^2-W(\rho_t,\si)^2)\\
&\le \limsup_{t\to 0}\frac{1}{2t}(W(\rho_t,\rho)^2+2W(\rho_t,\rho)W(\rho_t,\si))
\\&\le \limsup_{t\to 0}\frac{1}{2t}W(\rho_t,\rho)^2+\frac{1}{t}W(\rho_t,\rho)W(\rho_t,\si)
\\&\le W(\rho,\si)\sqrt{I(\rho_t)}.
\end{align*}
where in the last inequality we used Lemma \ref{conti} iii),
\begin{align*} &\lim_{t\to 0}W(\rho_t,\rho)=0\pl ,\pl  \lim_{t\to 0}W(\rho_t,\si)\le \lim_{t\to 0}W(\rho_t,\rho)+W(\rho,\si)=W(\rho,\si)\pl. \qedhere \end{align*}
\end{proof}

\begin{prop}\label{HWI2}For $\la>0$,  $\la$-HWI implies $\la$-MLSI.
\end{prop}
\begin{proof} Since
$ W(\rho,T_s(\rho))\le \int_0^sI(\rho_t)^{1/2}dt<\infty$, we can choose $\si=T_s(\rho)$ in HWI inequality for any $s>0$. By Lemma \ref{conti} (iv),
\[\lim_{s\to \infty}H(T_s(\rho))-H(E(\rho))=\lim_{s\to \infty}D(T_s(\rho)||E(\rho))=0\pl.\]
Then for any $\rho\in S_B(\A_0)$, we apply HWI inequality for $\si=T_s(\rho)$
\begin{align*}
D(\rho||\N)=&H(\rho)-H(E(\rho))=
H(\rho)-\lim_{s\to \infty}H(T_s(\rho))\\ \le& \lim_{s\to \infty} W(\rho,T_s(\rho))\sqrt{I(\rho)}-\frac{\la}{2}W(\rho,T_s(\rho))^2
\\ \le & \frac{1}{2\la}I(\rho)
\end{align*}
Here, in the last step we used the elementary inequality
\[xy\le cx^2+\frac{y^2}{c}\pl, \pl x,y,c>0\pl. \]
for $x=W(\rho,T_s(\rho)),y=I(\rho),c=\la/2$.
\end{proof}

\begin{rem}{\rm Here we can not choose $\si=E(\rho)$ because in general we do not know $W(\rho,E(\rho))<\infty$ for $\rho\in S_B(\A_0)$. In particular, the finite distance for $\rho\in S_H(\M)$ and $E(\rho)$ is a consequence of MLSI via the transport cost inequality (See \cite[Section 6]{CLSI}) as follows, \begin{align}W(\rho,E(\rho))\le \sqrt{\frac{2D(\rho||E(\rho))}{\la}}\pl. \label{TC}\end{align}
We call the above inequality \eqref{TC} $\la$-transport cost inequality or in short $\la$-TC.}
\end{rem}
Now we have two ways to reach Bakry-Emery Theorem.
\begin{cor}[Non-ergodic Bakry-Emery Theorem]
For $\la >0$, $\la$-GE implies $\la$-MLSI
\end{cor}
\begin{proof}We can either use
$\la$-GE $\Rightarrow$ $\la$-FM $\Rightarrow$ $\la$-MSLI or $\la$-GE $\Rightarrow$ $\la$-HWI $\Rightarrow$ $\la$-MSLI.
\end{proof}

Beyond positive curvature lower bound, we also have two ways for MLSI. The first one is
to apply our Theorem \ref{CLSI1} with the above discussion. Recall that the function $\kappa(\la,t)=\la(2-2e^{-2\la t})^{-1}$.
\begin{cor}\label{CLSI2}
Let $T_t:\M\to \M$ be a symmetric quantum Markov semigroup. Suppose
\begin{enumerate}
\item[i)] $T_t$ satisfies $\la$-GE for some $\lambda \in \mathbb R$;
\item[ii)] $T_t$ has finite CB-return time $t_{cb} < \infty$.
\end{enumerate}
Then $T_t$-satisfies $\kappa(\la,t_{cb})$-MLSI.
\end{cor}
\begin{rem}{\rm
Note that for $\la t>\ln\sqrt{2}$, $\kappa(\la,t)>\la$. This means when the CB-return time $t_{cb} < \la^{-1}\ln\sqrt{2}$, Corollary \ref{CLSI2} gives stronger MLSI-constant than Bakry-Emery Theorem. Also for $\la>0$, $\kappa(\la,t)\to \la/2$ when $t_{cb}\to \infty$.}
\end{rem}
One can compare the above corollary to the approach in \cite[Corollary 3.1]{OV} using the transport inequality.
\begin{cor}
Let $T_t:\M\to \M$ be a symmetric quantum Markov semigroup. Suppose
\begin{enumerate}
\item[i)] $T_t$ satisfies $\la$-GE for some $\lambda \in \mathbb R$;
\item[ii)] $T_t$ satisfies $\gamma$-transport cost inequality in \eqref{TC} for $\gamma\ge \max \{-\la,0\}$
\end{enumerate}
Then $T_t$-satisfies $\al$-MLSI for $\al=\max\{ \lambda, \frac{\gamma}{4}(1+\frac{\la}{\gamma})^2\}$.
\end{cor}
The proof is similar to \cite[Corollary 3.1]{OV}. One could also replace ``TC'' in condition ii) by the so called ``MLSI+TC'' inequality
\begin{align}W(\rho,E(\rho))\le \sqrt{\frac{I(\rho)}{\gamma}}\pl. \label{MLSITC}\end{align}
to obtain a similar estimate as in \cite[Corollary 3.2]{OV}.

\subsection{Bochner's Inequality}
We shall now discuss the curvature lower bound condition introduced in \cite{JLLR}.
 Let $T_t:\M\to \M$ be a symmetric quantum Markov semigroup and $(\A,\hat{\M},\delta)$ be a derivation triple for $T_t$. Denote $\Omega_\delta$ as the closure of $\A\delta(\A)$ in $L_2(\hat{\M})$. It follows from Leibniz rule that $\Omega_\delta$ is a $\A$-bimodule.
 To distinguish with the entropy Ricci curvature lower bound, we refer the following notion from \cite{JLLR} as geometric Ricci curvature lower bound.
\begin{defi}\label{defi}
We say $(\A,\hat{\M},\delta)$ satisfies a geometric Ricci curvature lower bound $\la$ for $\la\in \mathbb{R}$ (in short $\ARic\ge \la $) if there exists a symmetric quantum Markov semigroup $\hat{T}_t=e^{-\hat A{t}}:\hat{\M}\to \hat{\M}$ with generator $\hat{A}$ such that
\begin{enumerate}
\item[i)] $\hat{T}_t|_\M=T_t$ for any $t\ge 0$.
\item[ii)] $\delta(\A_0)\subset \dom(\hat{A})$ and there exists a $\A$-bimodule operator $\Ric: \Omega_\delta\to L_2(\hat{\M})$ such that for $x\in \A_0$,    \begin{align}\label{inter}\Ric(\delta(x))=\hat{A}\delta(x)-\delta A(x). \end{align}
\item[iii)] for any $y\in \Omega_\delta$, \begin{align}\lan y, \Ric(y)\ran\ge \la \lan y, y\ran\pl.\label{L2}\end{align}
    where $\lan \cdot,\cdot\ran$ is the trace inner product of $(\hat{\M},\tau)$.
\end{enumerate}
\end{defi}
We call the bimodule map $\Ric$ ``Ricci operator'' as an analog of Ricci tensor in geometry.
The above definition is of course an imitation of Bochner–Weitzenb\"ock–Lichnerowicz formula (c.f. pp374 \cite{OT})
\begin{align}\label{12} -\Delta+\nabla\nabla^*+\Ric=0 \pl.\end{align}
where $\Delta=\nabla^*\nabla$ is the Laplace-Beltrami operator on a Riemannian manifold and $\nabla$ is the gradient operator. When acting on a gradient $\nabla f$, \eqref{12} becomes
\[-\Delta (\nabla f)+\nabla (\Delta f)+\Ric(\nabla f)=0 \pl,\]
which is the motivation for \eqref{inter}. Note that the above Definition \eqref{defi} adds a little flexibility that $\hat{A}$ can be any generator extending $A$ on $\M$. We discuss more on the connection to classic Ricci curvature in Section \ref{heat}\\

On the other hand, we emphasize that Definition \ref{defi} is different from the entropy Ricci lower bound in Definition \ref{definition}. One major difference is that Definition \ref{defi} is automatically ``complete'' in the sense that if $T_t$ has $\ARic\ge \la $ (in our sense), then $T_t\ten \id_\mathcal{R}$ has $\ARic\ge \la $ for any finite von Neumann algebra $\R$. Indeed, both the algebraic equation \eqref{inter} and the $L_2$ inequality \eqref{L2} naturally extends to $T_t\ten \id_\mathcal{R}$. In contrast,
we will discuss in Section \ref{depolar} that the $2$-dimensional depolarizing semigroup has sharp entropy curvature lower bound by $1$, but $S_t\ten \id$ does not. This implies entropy curvature bound is not automatic complete.

We recall the following results from \cite{JLLR}.
\begin{theorem}[Theorem 3.6 of \cite{JLLR}]For $\la\in \mathbb{R}$, $T_t$ has $\ARic\ge \la$ implies that $T_t\ten \id_\R$ has $\la$-GE for any finite von Neumann algebra $\mathcal{R}$.
\end{theorem}
The next theorem is inspired by the discussion in \cite[Section 8.3]{CM} (see also \cite[Theorem 10.8]{CM18} and \cite[Proposition 5]{DR}).
\begin{theorem}\label{alg}Let $T_t:\M\to\M$ be a symmetric quantum Markov semigroup and let $(\A,\hat{\M},\delta)$ be a derivation triple of $T_t$. Suppose that there exists a symmetric quantum Markov semigroup $\hat{T}_t:\hat{\M}\to \hat{\M}$ such that for any $t\ge 0$,
\begin{align}\label{al} \pl \pl \tilde{T}_t|_\M=T_t\pl , \pl \text{and}\pl \pl \delta\circ T_t=e^{-\la t}\hat{T}_t\circ \delta\end{align}for some $\la\in \mathbb{R}$.
 Then $T_t$ satisfies $\ARic\ge \la$.  Moreover, the Ricci operator $\ARic$ can be taken to a constant multiple of the identity operator. 
\end{theorem}
\begin{proof} Let $\hat{A}$ be the generator of $\hat{T}_t$. For $x\in \A_0$,
\begin{align*}\lim_{t\to 0} \frac{1}{t}(e^{-\la t}\hat{T}_t( \delta(x))-\delta(x))=&
\lim_{t\to 0} \frac{e^{-\la t}}{t}(\hat{T}_t( \delta(x))-\delta(x))+\frac{1}{t}(e^{-\la t}\delta(x)-\delta(x))
\\=&\hat{A}\delta(x)-\la \delta(x)\pl.
\end{align*}
which converges in $w^*$-topology because $\delta(\A_0)\subset \dom(\hat{A})$. On the other hand, for $y\in \delta(\A_0)$ and $\delta^*\delta(y)=Ay\in L_2(\M)$,
\begin{align*} \lim_{t\to 0}\frac{1}{t}\big(\tau(y\delta(x))-\tau(y\delta(T_t(x))\big) )=&\lim_{t\to 0}\frac{1}{t}\big(\tau(\delta^*(y)x)-\tau(\delta^*(y)T_t(x)\big) )\\=&\tau(\delta^*(y)A(x))\pl.
\end{align*}
which implies $\displaystyle\lim_{t\to 0}\frac{1}{t}(\delta(T_t(x))-\delta(x))=\delta(A(x))$ weakly. Thus we have for $x\in \A_0$,
\[ \delta(A(x))=\hat{A}\delta(x)-\la \delta(x)\pl.\]
which means the Ricci operator is constant $\Ric(\delta(x))=\lambda \delta(x)$. 
\end{proof}
As we see in the above proof, the relation \eqref{al} is equivalent to the Ricci operator in \eqref{inter} equaling to a multiple of the identity. We emphasize this special case by giving the following definition.
\begin{defi}We say a semigroup $T_t$ satisifies constant $\la$-Ricci curvature condition ($\la$-$\ARic$) if $T_t$ admits a derivation triple satisfying \eqref{al}.
\end{defi}
We remark that the $\la$-$\ARic$ relation deos not gives the meaning that Ricci curvature is constant $\la$ but still just a lower bound by $\la$. We revisit the Orstein-Unlenbeck semigroup discussed in \cite{CM}.
\begin{exam}{\rm Let $\mathbb{R}^n$ be the $n$-dimensional real Euclidean space and $\mu$ the standard Gaussian distribution.
The Orstein-Unlenbeck (OU) semigroup $T_t=e^{-At}:L_\infty(\mathbb{R}^n, \mu)\to L_\infty(\mathbb{R}^n, d\mu)$ is given by
\[ T_tf (x)=\int_{\mathbb{R}^n} f(e^{-t}x+\sqrt{1-e^{-2t}}y)d\mu(y) \pl, \pl.\]
Denote $\partial_j=\frac{\partial}{\partial x_j}$ be the partial derivative. The generator of the OU semigroup is given by  \[A =\Delta+x\cdot\nabla=\nabla^*\nabla +x \cdot \nabla =\sum_{j=1}^n(-\partial_j^2+x_j\partial_j).\]
Consider the derivation \[\delta:C^\infty(\mathbb{R}^n)\to \oplus_{j=1}^n C^\infty(\mathbb{R}^n)\pl, \pl \delta(f)=(\partial_j f)_{j=1}^n\pl.\] As observed in \cite[Section 8.1]{CM}, we have the relation
$[\partial_j, -\Delta+x\cdot\nabla]=\partial_j$ for $j=1, \ldots, n$. This translates to the equality
\[           (A\ten \id)\circ \delta -A\circ \delta= \delta\pl,     \]
where $\hat A = A\ten \id$ is the extension of $A$ to $\oplus_{j=1}^n C^\infty(\mathbb{R}^n)\cong C^\infty(\mathbb{R}^n)\ten l_\infty^n$, which is clearly the generator of the semigroup $\hat T_t = T_t\ten \id$ on $L_\infty(\mathbb{R}^n)\ten l_\infty^n$.  In particular, this gives a derivation triple for the OU semigroup that satisfies $1$-$\ARic$. Moreover since $T_t$ has spectral gap $1$, we can therefore conclude the sharp complete version result that $T_t\ten \id_\R$ satisfies $1-GE$ for any finite von Neumann algebra $\R$, and $T_t$ satisfies $1$-CFM and $1$-CLSI}
\end{exam}

We have a complete version of Corollary \ref{CLSI2}
\begin{cor}\label{CLSI3}
Let $T_t:\M\to \M$ be a symmetric quantum Markov semigroup. Suppose
\begin{enumerate}
\item[i)] $T_t$ satisfies $\ARic\ge \la$ for some $\lambda \in \mathbb R$;
\item[ii)] $T_t$ has finite CB-return time $t_{cb} < \infty$.
\end{enumerate}
Then $T_t$-satisfies $\kappa(\la,t_{cb})$-CLSI.
\end{cor}

\section{Examples}
In this section, we discuss applications to classical Markov semigroups and finite dimensional quantum Markov semigroups.

\subsection{Diffusion Semigroups}
Our motivation for Fisher monotonicity was from Bakry-Emery's curvature dimension condition for diffusion Markov semigroup.
We refer to \cite{BGL} for more information on classical diffusion Markov semigroup.

Let $(\Omega,\mu)$ be a Borel space equipped with a Borel probability measure $\mu$. Let  $T_t:L_\infty(\Omega,\mu)\to L_\infty(\Omega, \mu)$ be an ergodic Markov semigroup and $A$ be its generator.
We say $T_t$ satisfy \emph{diffusion property} if its gradient form $\Gamma$ satisfies the following product rule,
\begin{align} \label{diff} \Gamma(fh,g)=f\Gamma(h,g)+h\Gamma(f,g)\pl.\end{align}
Denote $\Gamma(f):=\Gamma(f,f)$.
It then follows from polynomial approximation that for a smooth function $\psi:\mathbb{R}\to \mathbb{R}$,
\begin{align*}
\Gamma(\psi(f),g)=\psi'(f)\Gamma(f,g)\pl,  \Gamma(\psi(f))=\psi'(f)^2\Gamma(f,g)\pl
\end{align*}
For a density function $f\in L_\infty(\Omega,\mu)$,
the entropy $H(f)$ (also called Boltzman $H$-functional) and the Fisher information $I(f)$ are given by
\begin{align*} &H(\rho)=D(\rho||1)=\int_\Omega \rho \log \rho \pl d\mu\\
& I(f)=-\int (A f)\log f d\mu =\int \Gamma(f,\log f) d\mu=\int f\Gamma(\log f) d\mu
\end{align*}
Recall that the $\Gamma_2$ operator is defined as
\[ \Gamma_2(f,g)=\frac{1}{2}\Big(\Gamma(Af,g)+\Gamma(f,Ag)-A\Gamma(f,g)\Big)\pl.\]
Denote $\Gamma_2(f):=\Gamma_2(f,f)$.  $\Gamma_2$ can be realized as
\[\Gamma_2(f)=\lim_{t\to 0} \frac{T_t(\Gamma(f))-\Gamma(T_t(f))}{t}\pl,\]
The derivative of Fisher information is
\begin{align}\label{d2}
\frac{d I(T_tf)}{dt} =-2\int T_tf\Gamma_2(\log T_tf)d\mu\pl.
\end{align}
Recall that $T_t$ satisfies $(\la,\infty)$-curvature dimension condition for $\la\in \mathbb{R}$ (in short, CD($\la$,$\infty$)) if for any $f\in \dom(A)$
\[\Gamma_2(f)\ge \la \Gamma(f)\pl.\]
It follows immediately $CD(\la,\infty)$ implies $\la$-FM. For $\la>0$, it is the Barky-Emery theorem that $CD(\la,\infty)$ $\Rightarrow$ $\la$-FM $\Rightarrow$ $\la$-MLSI.
For general $\la\in \mathbb{R}$, we have the following theorem for diffusion Markov semigroups.
\begin{theorem}\label{MLSI}
Let $T_t:L_\infty(\Omega,\mu)\to L_\infty(\Omega, \mu)$ be an ergodic symmetric diffusion Markov semigroup. Suppose $T_t$ satisfies curvature-dimension condition $CD(\la,\infty)$. If in addition, we assume
 \begin{enumerate}
 \item[i)] $\|T_t:L_1(\Omega)\to L_{\infty}(\Omega)\|\kl c t^{-d/2}$ for some $c,d>0$ and all $0< t< 1$;
 \item[ii)] the generator $A$ satisfies spectral gap $\si>0$.
 \end{enumerate}
Then $T_t$-satisfies $m(\la)$-MLSI for
\begin{align*}
m(\la)=\begin{cases}
 \Big(2+2(d-1)
 \log 2+\frac{4}{\si}\log c\Big)^{-1}, & \mbox{if } \la=0 \\
\la\Big({2-2^{1-(d-1)\la}
 c^{-\frac{2\la}{\si}} }\Big)^{-1}     , & \mbox{if } \la \neq0.
\end{cases}
\end{align*}
\end{theorem}
\begin{proof}The condition i) is the Varopoulos' dimension condition. Here the CB-norm estimate is automatic:
\[\|T_t-E_\tau:L_1(\Omega)\to L_{\infty}(\Omega)\|=\|T_t-E_\tau:L_1(\Omega)\to L_{\infty}(\Omega)\|_{cb}.\]
This is because $L_{\infty}(\Omega)$ is a commutative space (see \cite[Proposition 1.10]{pisieros}). The assertions follows from Theorem \ref{CLSI1} and the return time estimates in Lemma \ref{sfc}.
\end{proof}

 \begin{rem}{\rm It is well known that if $\displaystyle T_tf(x)=\int_{\Omega} k_t(x,y)f(y)d\mu$ is given by the kernel function $k_t(x,y)$. Then
\[ \norm{T_t-E_\tau:L_1(\Omega)\to L_{\infty}(\Omega)}{}=\norm{k_t-1}{\infty}\pl.\]
 is a kernel estimate.}
 \end{rem}

\subsection{Heat semigroups}\label{heat}

We shall now discuss the heat semigroups. We refer to \cite{OT} for more information on analysis of heat semigroups on manifolds.
 Let $(M,g)$ be a complete compact Riemannian manifold equipped with Riemannian metric $g$.  Let $\Delta$ be the Laplace-Beltrami operator given by
\[\Delta f= \nabla^* \nabla f\pl.\]
where $\nabla$ is the gradient operator and $\nabla^*=\text{div}$ is the divergence. The heat semigroup $T_t=e^{-\Delta t}: L_\infty(\M,d \text{vol})\to L_\infty(\M,d \text{vol})$ is a Markov semigroup with respect to the volume form $d\text{vol}$ induced by $g$. Recall the Bochner–Weitzenb\" ock–Lichnerowicz formula that for the vector field $\nabla \phi$,
\[-\frac{1}{2}\Delta |\nabla \phi|^2+ \nabla \phi\cdot \nabla (\Delta\phi)+\norm{\nabla\phi}{2}^2+\Ric(\nabla \phi,\nabla\phi)=0,\]
which translates to
\begin{align}-\Delta+\nabla\nabla^*+\Ric=0\pl. \label{bochner}\end{align}
The $C^\infty(\M)$-bimodule property of $\Ric$ is exactly the fact that the Ricci curvature is a smooth tensor over $M$.

The same argument applies to weighted Riemannian manifolds $(M,g,e^{-W}d\text{vol})$ where $e^{-W}$ is a smooth density function with respect to $d\text{vol}$. The weighted Laplacian is
\[ \Delta_W=\nabla^*\nabla=\Delta-\nabla W\cdot \nabla \pl.\]
where $\nabla^*$ is adjoint of $\nabla$ with respect to $L_2(M,e^{-W}d\mu)$ and $\Delta_W$ is a self-adjoint operator on ${L_2(M,e^{-W}d\mu)}$. Then the weighted heat semigroup $T_t=e^{-\Delta_W t}$ is an ergodic symmetric Markov semigroup with the unique invariant measure $e^{-W}d\mu$. In this case,
\[ \Delta_W-\nabla \nabla^*=\text{Ric}_W \pl.\]
where $\text{Ric}_W=\text{Ric}_g+\nabla\nabla W$ is the sum of Ricci curvature tensor of the metric $g$ and the Hessen of the function $W$. The weighted Ricci curvature bound $\text{Ric}_W\ge \la$ is that $\Ric_W(\xi,\xi)\ge \la g(\xi,\xi)$ for any vector field $\xi\in TM$. When $\la>0$, $\text{Ric}_W\ge \la$ implies $T_t=e^{-\Delta_W t}$ satisfies $\la$-MLSI by the Bakry-Emery Theorem.

It is proved in \cite[Section 4]{JLLR} that $\text{Ric}_W\ge \la$ actually implies $\ARic\ge \la$, which implies a complete version of Bakry-Emery theorem.

\begin{theorem}[\cite{JLLR}]\label{JLLR}
If $\Ric_W(\xi,\xi)\ge \la g(\xi,\xi)$ for any $\xi\in TM$, then
the weighted heat semigroup $T_t=e^{-\Delta_W t}$ satisfies $\ARic \ge \la$. In particular, if $\Ric_W\ge\la>0$, $T_t=e^{-\Delta t}$ satisfies $\la$-CLSI.
\end{theorem}

The proof uses the Clifford bundle $Cl(M)$ as the quantization of tangent bundle $TM$. Then the $\ARic
\ge \la$ is a realization of the Bochner identity on $Cl(M)$. We refer to \cite{JLLR} for details.

Now we apply our method for general compact weighted manifolds. It follows from compactness and continuity that $\text{Ric}_W\ge\la$ always holds for some real $\la$. Indeed, for each $x \in M$, $\text{Ric}_W$ at $x$ is a real symmetric matrix with respect to an orthonormal basis of $g$. Hence
\[(\text{Ric}_W)_x\ge \la_{min}(x) g\ge \min_{x\in M}\la_{min}(x)g\]
Here $\la_{min}(x)$ is the smallest eigenvalue of $(Ric_W)_x$ with respect to metric $g$,  which is continuous depending on $x\in M$. Define that $\Ric(\Delta_W)=\min_{x\in M}\la_{min}(x)$ as the global minimum of $\la_{min}(x)$.
Thus the heat semigroup $T_t=e^{-\Delta_W t}$ always satisfies
$\ARic\ge \la$ for some real $\la=\Ric(\Delta_W)$. The following is an application of Theorem \ref{CLSI3}.
\begin{theorem}\label{riemann}Let $(M,g,e^{-W}d\text{vol})$ be a compact connected weighted Riemannian manifold.
Then the weighted heat semigroup $T_t=e^{-\Delta_W t}$ satisfies $\la$-CLSI for some $\la>0$.
\end{theorem}
\begin{proof}
We know from Theorem \ref{JLLR} that $T_t=e^{-\Delta_W t}$ always satisfies
$\ARic\ge \Ric(\Delta_W)\in \mathbb{R}$. On the other hand, both spectral gap and finite Varopoulos dimension of $\Delta_W$ are well-known for compact weighted manifolds.  See \cite[Theorem 10.23]{heat} for spectral gap and \cite[Theorem 14.19 \& Exercise 15.2]{heat} for Varopoulos dimension. Indeed, the $T_t=e^{-\Delta_W t}$ satisfies the ultra-contractive estimates of dimension $n=\text{dim}(M)$,
\[ \norm{T_t: L_1(M,d\text{vol})\to L_\infty(M,d\text{vol})}{}\le ct^{-n/2}\pl, \pl 0<t\le 1, \pl.\]
Then it follows from Lemma \ref{sfc} and Corollary \ref{CLSI3} that $T_t=e^{-\Delta_W t}$ satisfies $\la$-CLSI
where $\la$ is determined by $\Ric(\Delta_W)$, spectral gap of $\Delta_W$ and the ultra-contractive estimate of $e^{-\Delta_W t}$.
\end{proof}

The above theorem has the following refined form.
\begin{theorem}\label{riemann1}
Let $(M,g)$ be a connected compact Riemannian manifold and let $\Delta$ be the Laplace-Beltrami operator. Suppose the Ricci curvature of $M$ is bounded below by $K$ for some $K\in \mathbb{R}$.
\begin{enumerate}
\item[(i)] the heat semigroup $T_t=e^{-\Delta t}$ satisfies $\la$-CLSI for
\begin{align}\label{estimate}
\la= \begin{cases}
                K, & \mbox{if } K>0 \\
                \Big(4+\frac{4}{\si}\log (2C_1)\Big)^{-1}, & \mbox{if } K=0 \\
              K\Big(2-e^{-2K} (\frac{2c(K,n)}{V})^{-\frac{2K}{\si}}\Big)^{-1} , & \mbox{if }K<0.
              \end{cases}
\end{align}
 where $\si$ is the spectral gap of $\Delta$, $V$ is the minimum volume of radius $1$ ball in $M$, $C_1$ is a universal constant and $C_2(K,n)$ only depends on $K$ and the dimension $n=\dim(M)$.
\item[(ii)] Let $W$ be a smooth function on $M$ such that $e^{-W}$ is a probability density function for the volume form $d${\rm vol}. Then the weighted heat semigroup $T_t=e^{-\Delta_W t}$ satisfies $c\la$-CLSI where $\la$ is given in \eqref{estimate} and $c=e^{\min{W}-\max{W}}$.
\end{enumerate}
\end{theorem}
\begin{proof}The case $K>0$ is in Theorem \ref{JLLR}. We argue for the case $K\le 0$.
Denote $k:M\times M\times \mathbb{R}_+\to \mathbb{R}$ as the heat kernel. Recall the famous Li-Yau estimate  that for a complete Riemannian manifold with Ricci curvature bounded below by $\Ric(M)\ge -K$ for some $K \ge 0$, the heat kernel satisfies
\[k(x,y,t)\le \frac{C_1}{\sqrt{V(x,\sqrt{t})V(y,\sqrt{t})}}\exp \Big(C_2Kt-\frac{d(x,y)^2}{5t}\Big)\pl.\]
where $d(x,y)$ is the Riemannian distance, $V(x,\sqrt{t})$ is the volume of geodesic ball center at $x$ with radius $\sqrt{t}$,  $C_1$ is some universal constant and $C_2$ only depends on the dimension $\dim(M)=n$. (We choose the parameter $\epsilon=1$ in statement of \cite[Corollary 3.1]{LY}). On diagonal $x=y$, we have
\[k(x,x,t)\le \frac{C_1}{V(x,\sqrt{t})}\exp \big(C_2Kt\big)\pl.\]
Take $V=\min_{x\in M} V(x,1)$ as the minimum volume of radius $1$ ball in $M$. Then for $t=1$,
\begin{align*}k(x,x,1)\le C_1 V(x,1)^{-1} \exp\Big(C_2Kt\Big)\le c(K,n) V^{-1}, \end{align*}
where $c(K,n)=C_1\exp (C_2K)$ is a constant only depending on $\dim(M)=n$ and curvature bound $K$ (for $K=0$, $C(0,n)$ is also independent of $n$).
The ultra-contractive estimate is given by heat kernel on the diagonal,
 \[ \norm{T_1:L_1(M,d\text{vol})\to L_\infty(M,d\text{vol})}{}=\sup_x k(x,x,1)\le c(K,n)V^{-1} \pl.\]
Let $\si$ be the spectral gap of $\Delta$. By Lemma \ref{esti}, we have
 \[t_{cb}\le 1+\frac{1}{\si}\log (2c(K,n)V^{-1})\]
The assertion follows from Corollary \ref{CLSI3}. This proves i). ii) follows from the change measure \cite[Lemma 2.11]{JLLR}. Indeed, for smooth (operator-valued) function $f$
\[ I_{\Delta_W}(f)=\int \lan \nabla f, \nabla \log f\ran e^{-W}d\text{vol}\ge e^{-\max W}\int \lan \nabla f, \nabla \log f\ran d\text{vol}=I_{\Delta}(f)\pl,\]
where $I_{\Delta}$ is the Fisher information for the standard Laplacian and $I_{\Delta_W}$ for the weighted Laplacian $\Delta_W$. The comparison for relative entropy follows from \cite[Lemma 2.8]{JLLR}.
\end{proof}
\subsection{Central semigroups on compact groups}\label{liegroup}
In this subsection, we consider Markov semigroups on compact groups. Let $G$ be a compact group. We denote by $C(G)$ (resp. $C^\infty(G)$) the space of continuous (resp. smooth) functions on $G$ and denote by $L_\infty(G)=L_\infty(G,m)$ the  $L_\infty$-space with respect to the Haar probability measure $m$. Let $L_g:L_\infty(G)\to  L_\infty(G)$ (resp. $R_g$) be the left (resp. right) translation operator.
\begin{align*}
(L_gf)(h)=f(gh)\pl,\pl (R_gf)(h)=f(hg)\pl.
\end{align*}We say a Markov semigroup $T_t:L_\infty(G)\to L_\infty(G)$ is {\it left (resp. right) invariant} if $L_g\circ T_t=T_t\circ L_g$ (resp. $R_g\circ T_t=T_t\circ R_g$) for all $g\in G$.
We say $T_t$ is \emph{central} if it is both left and right invariant. Recall that a function $k \in L_1(G)$ is {\it central} if $k(sgs^{-1})=k(g)$ for a.e. $g,s\in G$.  This is equivalent to the condition $f\star k = k\star f$ for all $f \in L_1(G)$, where $\star$ denotes the convolution product on $L_1(G)$.  We denote the subalgebra of central functions in $L_1(G)$ by $ZL_1(G)$.   It is well known that a Markov semigroup $T_t$ on $L_\infty(G)$ is central if and only if there exists a {\it convolution semigroup} of central probability densities $(k_t)_{t\ge 0} \subset ZL_1(G)$
\[
T_t f (g) =(f \star k_t) (g) = \int_G f(y)k_t(y^{-1}g)dm(y), \ f \in L_\infty(G).
\]

Now consider the the co-multiplication map $\al:L_\infty(G, m)\to L_\infty(G\times G,m\times m)$,
\begin{align*} \al(f)(g,h)=f(gh)\pl, \pl \al(f)(g,\cdot)=L_{g}f\pl,\al(f)(\cdot,h)=R_{h}f \end{align*}
It is clear that $\al$ is a $m$ to $ m\times m$ measure preserving $*$-monomorphism. Moreover, if $T_t$ is a left invariant semigroup we have the commution relation $\al\circ T_t =(\id\ten T_t)\circ \al$. Indeed,
\[\al(T_tf)(g,\cdot)=L_g(T_t f)=T_t(L_g f)=\id\ten T_t(\al(f))(g,\cdot)\]
Similarly, if $T_t$ is right invariant, we have $\al\circ T_t =(T_t\ten \id)\circ \al$. Thus for a central semigroup $T_t$, we have the following commutative diagram
\begin{equation}\label{ccss}
 \begin{array}{ccc}  L_{\infty}(G\times G)\pl\pl &\overset{ \operatorname{id}_G\ten T_t \pl \text{or}\pl T_t \ten \operatorname{id}_G}{\longrightarrow} & L_{\infty}(G\times G) \\
                    \uparrow \al    & & \uparrow \al  \\
                     L_\infty(G)\pl\pl &\overset{T_t}{\longrightarrow} & L_\infty(G)
                     \end{array} \pl .
                     \end{equation}
                     This is a crucial point in the following lemma.
       \begin{lemma}\label{group}Let $G$ be a compact group and $T_t:L_\infty(G)\to L_\infty(G)$ be a central Markov semigroup. Then $T_t$ satisfies $\ARic\ge 0$ and hence complete Fisher monotonicity .
\end{lemma}
\begin{proof}Let $A$ be the generator of $T_t$ and $(\A_\E,\M,\delta)$ be a derivation triple for $T_t$. That is, $\delta:\A_\E\to L_2(\M)$ is a $*$-preserving derivation such that
\[E(\delta(x)^*\delta(y))=\Gamma_A(x,y)\pl.\]
where $E$ is the conditional expectation on to $L_\infty(G) \subseteq \M$, and $\A_\E= L_\infty(G)\cap \dom(A^{1/2})$ is the Dirichlet subalgebra. We show that \[\partial=(\delta\ten \id)\circ \al:L_\infty(G)\to L_\infty(G ,\M) \cong \M \bar \otimes L_\infty(G)\] is also a derivation for $T_t$. Let $E_\al: L_\infty(G\times G)\to L_\infty(G)$ be the conditional expectation obtained as the adjoint of $\al$.  Using the commutative diagram \eqref{ccss}, we have $E_\al(A\ten \id)\al=A$, which follows by  differentiating $\al\circ T_t =(T_t\ten \id)\circ \al$. Then for the gradient forms associated to $A$ and $A \otimes \id$ (the latter which acts on $\alpha(\A_\E) \subset \alpha(L_\infty(G))$), we have
 \begin{align*} \Gamma_A(x,y)&=x^*Ay+(Ax)^*y-A(x^*y)
 \\&=x^*E_\al(A\ten \id)\al(y)+(E_\al(A\ten \id)\al(x))^*y-E_\al(A\ten \id)\al(x^*y)
 \\&=E_\al(\al(x)^*(A\ten \id)\al(y)+(A\ten \id)\al(x)^*\al(y)-(A\ten \id)\al(x^*y))
 \\ &=E_\al(\Gamma_{A \otimes \id}(\al(x),\al(y)))\\ &=E_\al\circ (E\otimes \id)((\delta\ten id)\al(x)^*(\delta\ten id)\al(y)) \\
&=E_\al\circ (E\otimes \id)(\partial(x)^*\partial(y))
\end{align*}
where we have used the fact
 $(\delta\ten \id)$ is a derivation for $T_t\ten \id$. Here $E_\al\circ (E \otimes \id)$ is exactly the conditional expectation onto $\al(L_\infty(G)) \subset L_\infty(G,\M)$. Thus we have shown that $(\A_\E, L_\infty(G,\M),\partial)$ is a new derivation triple for $T_t$. Now for this derivation, we have
  \begin{align*}\partial\circ T_t=&(\delta\ten \id_G)\circ\al\circ T_t=(\delta\ten \id_G)(\id_G\ten T_t)\circ \al\\= &(\id_\M\ten T_t)(\delta\ten \id)\circ \al=(\id_\M\ten T_t)\partial\pl.\end{align*}
 where $\id_G\ten T_t$ (resp. $\id_\M\ten T_t$) is the extension semigroup of $T_t$ on  $L_\infty(G\times G)$ (resp. $\M\overline{\ten} L_\infty(G)$). Note that here we used the other part of \eqref{ccss} $\al\circ T_t=(\id_G\ten T_t)\circ \al$ by the right invariance of $T_t$. This verifies the algebraic relation in Theorem \ref{alg} for $\la=0$, which implies the assertions.
\end{proof}

\begin{exam}[Heat semigroups]{\rm
Let $G$ be a compact Lie group and $\mathfrak{g}$ be its Lie algebra of left invariant vector fields. Let $X=\{X_1,...,X_r\}$ be an orthonormal basis of $\mathfrak{g}$ with respect to its Killing form. We consider the heat semigroup $T_t=e^{-\Delta t}$ generated by the Casimir operator $\Delta=\sum_{j}X_r^2$. The natural derivation for $\Delta$ is the gradient \[\nabla: C^\infty(G)\to \oplus_{j=1}^r C^\infty(G)\pl , \nabla(f)=(X_jf)_{j=1}^r\pl\]
It is known from representation theory that $\Delta=\sum_{j}X_j^2$ as a generator is central. Indeed, recall that for an irreducible continuous representation $\pi: G\to B(H_\pi)$ on the Hilbert space $H_\pi$, the coefficient function space associated to $\pi$ is the finite-dimensional subspace
\[\E_\pi(G)=\{ \pl f \in C(G) : f(g)=\lan h_1, \pi(g)h_2\ran_{H_\pi}\pl | \pl h_1,h_2\in H_\pi\pl \}\subset L_2(G)\pl.\]
Denote $E_\pi$ as the Hilbert projection from $L_2(G)$ to the closure of $\E_\pi(G)$.
The Casimir operator $\Delta$ then admits a spectral decomposition of the form
\[\Delta=\sum_{\pi\in Irr(G)}\la_\pi E_\pi\]
where the summation is over all irreducible representation $\pi$ and $\la_\pi$ is the common eigenvalue for all coefficient functions of $\pi$. Since the $E_\pi$ is invariant for both left translation and right translation, this implies $\Delta$ and the semigroup $e^{-\Delta t}$ are central.
By the construction in Theorem \ref{group}, the algebraic relation curvature relation $0$-$\ARic$ is satisfied with the following alternative derivation
\[\partial: C^\infty(G)\to \oplus_{j=1}^r C^\infty(G\times G)\pl, \partial f= (\nabla\ten \id)\al(f)(g,h)=(X_j f(gh))_{j=1}^r\pl. \]
Combined with the heat kernel estimate and spectral gap (see e.g. \cite{VSCC}), we have the following corollary.}\end{exam}

\begin{theorem}Let $G$ be a compact Lie group and let $\Delta$ be the Casimir operator.
For $r\in (0,1]$, denote $T_t^r=e^{-\Delta^r t}:L_\infty(G)\to L_\infty(G)$ as the heat semigroup ($r=1$) and its subordinated semigroup ($0<r<1$).  Then for each $r\in (0,1]$, $T_t^r$ satisfies $\ARic\ge 0$, complete Fisher monotonicity, and $\la(r)$-CLSI for \begin{align*}
\la(r)=\Big(4+4\si^{-r}\log (2c(r,n)+\frac{C}{V_1})\Big)^{-1}\pl.
\end{align*}
where $C$ is an absolute constant, $c(r,n)$ is a constant only depending on $0<r\le 1$ and $n=\dim(G)$, $\si$ is the spectral gap of $\Delta$ and $V_1$ is the volume of unit geodesic ball.
\end{theorem}
\begin{proof}For all $ r\in (0,1]$, $\Delta^r=\sum_{\pi} \la_\pi^rE_\pi$ is a central generator. Thus $T^r_t$ are central semigroup hence has $0$-$\ARic$. It is well-known (see e.g. \cite{VSCC}) that the heat semigroup $T_t^1=e^{-\Delta t}$ has ultra-contractive estimate
\[\norm{T_t:L_1(G,m)\to L_\infty(G,m)}{}=Ct^{-\frac{n}{2}}\pl, 0<t\le 1\]
where $n=\dim(G)$. By the discussion in \cite[Section II.3]{VSCC}, the subordinated semigroup $T_t^r$ has spectral gap $\si^r$ and Varopoulos dimension $\frac{1}{r}\dim(G)$.
Then assertions follows from Theorem \ref{CLSI3}.

We now give the concrete ultra-contractive estimates of $T^r_t$ for each $r$. Let $V_{t}$ be the volume of geodesic ball of radius $t$. Since $G$ has nonnegative Ricci curvature,
by Bishop-Gromov volume comparison theorem (c.f. \cite[Theorem 5.6.4]{SC}), for $0<t\le 1$, $V(t)\ge V(1)t^{n}$. Then for $r=1$ and $T_t:=T^r_t$, using the Li-Yau estimate \cite[Corollary 3.1]{LY} again,
\[ \norm{T_t:L_1(G)\to L_\infty(G)}{}= k(x,x,t)\le \frac{C}{V_{\sqrt{t}}}\le \frac{C}{V_{1}}t^{-\frac{n}{2}}\pl.\]
where $C$ is some absolute constant,
 $k(x,y,t)$ is the heat kernel of $T_t$, $x$ is some point in $G$, and $V_{\sqrt{t}}$ (resp. $V_{1}$) is the volume of geodesic ball in $G$ with radius $\sqrt{t}$ (resp. $1$). Denote $C_1=C/V_1$. For the subordinated semigroup, we the use the argument from \cite[Section II.3]{VSCC},
 \[T_t^r=e^{-\Delta^\al t}=\int_0^\infty f_{\al}(v)T_{vt^{1/\al}}dv\pl.\]
 where $f_\al$ is the function whose Laplace transform is $s\mapsto e^{-s^\al}$. In particular, $f_{\al}\ge 0$ and $\int_{0}^\infty f_\al(v)dv=1$\pl. Then for $t=1$,
 \begin{align*}
 &\norm{T_1^r:L_1(G)\to L_\infty(G)}{}\\ \le & \int_0^\infty f_{\al}(v) \norm{T_{v}:L_1(G)\to L_\infty(G)}{}dv\\
 \le &\int_0^{1} f_{\al}(v) \norm{T_{v}:L_1(G)\to L_\infty(G)}{}dv+
 \int_{1}^{\infty} f_{\al}(v)\norm{T_{v}:L_1(G)\to L_\infty(G)}{}dv\\
 \le &\int_0^{1} f_{\al}(v) v^{-\frac{n}{2}}dv+
 \int_{1}^{\infty} f_{\al}(v)\norm{T_{1}:L_1(G)\to L_\infty(G)}{}dv
 \\
 \le & \int_0^{1} f_{\al}(v) v^{-n/2}dv+
 C_1\int_{1}^{\infty} f_{\al}(v)dv
 \\
  \le & c(\al,n)+C_1
 \end{align*}
 where $c(\al,n)=\int_0^{1} f_{\al}(v) v^{-n/2}dv\le \int_0^{\infty} f_{\al}(v) v^{-n/2}dv<\infty$. By Lemma \ref{esti}, we have
 \[t_{cb}\le 1+\si^{-\al}\log \Big(2 c(\al,n)+\frac{C}{V_1}\Big)\pl. \]
The assertion follows from Corollary \ref{CLSI3}
\end{proof}

\begin{rem}{\rm a) In \cite[Section 7]{Milnor} Milnor proved that for any bi-invariant metric on $G$, the Ricci curvature is non-negative. Theorem \ref{group} recovers the non-negativity of Ricci curvature for all heat semigroups with bi-invariant metric. Furthermore, it also applies to subordinated semigroup beyond the Laplacian case.\\
b) Based on the derivation of heat semigroup $T_t=e^{-\Delta t}$, derivation triple for subordinated group can be constructed as in \cite[Section 10.4]{CS03}. Note that the CLSI of subordinate semigroup was obtained in \cite{CLSI} using a completely different method.\\
c) By Theorem \ref{riemann1}, the constant for the heat semigroup $T^1_t$ has the following explicit form
\[\la(1)=\Big(2+2(n-1)
 \log 2+\frac{4}{\si}\log (\frac{C}{V})\Big)^{-1}\]
 where $\si$ is the spectral gap, $V$ is the volume of unit ball and $C$ is some absolute constant.
}
\end{rem}
It was also pointed out in \cite{Milnor} that Ricci curvature of a left invariant metric is strictly positive if the fundamental group of $G$ is finite. It means for semi-simple Lie groups Theorem \ref{JLLR} usually gives better CLSI constant than Theorem \ref{group}.
Nevertheless, for non semi-simple Lie group with zero curvature lower bound, Theorem \ref{group} gives us an effective way to obtain lower bounds of CLSI constant.

\begin{exam}[Circle]{\rm \label{circle} Let $\mathbb{T}=\{z\in \mathbb{C}\pl |\pl |z|=1\}$ be the unit circle. Then $\{z^n | n\in \Z \}$ is a orthonormal basis of $L_2(\mathbb{T})$. The heat semigroup is given by
\[ T_t(z^m)=e^{-m^2 t}z^m, \]
and the associated heat kernel is given by $k_t(z) = \sum_{m \in \mathbb Z} e^{-m^2t}z^m.$
Now we estimate the cb-return time of $T_t$:
\begin{align*}
\norm{T_t-E_\tau:L_1(\mathbb{T})\to L_\infty(\mathbb{T})}{}&=\norm{\sum_{m\in \mathbb Z\backslash \{0\}}e^{-m^2t}z^mw^{-m}}{L_\infty(\mathbb T^2)}\\
&=\norm{\sum_{m\in \mathbb Z\backslash \{0\}}e^{-m^2t}z^m}{L_\infty(\mathbb T)}\\
&= \norm{k_t-1}{L_\infty(\mathbb T)}\\
&= k_t(e)-1\\
&= 2\sum_{m=1}^{\infty}e^{-m^2t}.
\end{align*}
In the above, the first equality follows from the isometric identification
\[ L_\infty(\mathbb T^2) \cong  L_\infty(\mathbb T) \bar \otimes L_\infty(\mathbb T) \cong B(L_1(\mathbb T), L_\infty(\mathbb T)); \quad (\varphi \otimes \psi)(f) = \Big(\int_\mathbb T \psi(w) f(w)dw\Big) \varphi.   \] The third equality follows from the fact that $k_t$ is a positive definite function on $\mathbb T$.
Denote $f(t)=2\sum_{m=1}^\infty e^{-m^2t}$, so that
\[t_{cb}= \inf\{t| f(t)\le 1/2\}\pl.\]
Using standard heat kernel estimates, we have
\[
2e^{-t} \le f(t) = k_t(0)-1 \le\frac{2e^{-t}}{1 - e^{-t}}\qquad (t >1).
\]
These estimates yield concrete bounds of the form
\[
1.38629 \sim \ln 4 \le t_{cb} \le \ln 5 \sim 1.60944.
\]
Numerical calculation shows that $t_{cb}\le 1.41 < 1.5$, and therefore the heat semigroup on $\mathbb{T}$ has $\displaystyle\frac{1}{6}$-CLSI.
}\end{exam}

\begin{exam}[$d$-Torus]\label{dtorus} {\rm Let $\mathbb{T}^d=\{z=(z_1,z_2,\cdots,z_d)\in \mathbb{C}^d\pl |\pl |z_i|=1, i=1,\cdots,d\}$ be the $d$-Torus. For a multi-index $m=(m_1,\cdots, m_d)\in \mathbb{Z}^d$, write $|m|^2=m_1^2+m_2^2+\cdots +m_d^2$ and
define the polynomials $z^m:=z_1^{m_1}z_2^{m_2}\cdots z_d^{m_d}$. The set $\{z^m | m\in \Z^d \}$ is an orthonormal basis of $L_2(\mathbb{T})$. The heat kernel $k_t^{(d)}$ and heat semigroup $T_t$ on $\mathbb{T}^d$ are given by
\[  k_t^{(d)}(z) = \sum_{m \in \mathbb Z^d} e^{-|m|^2 t}z^m, \quad T_t(z^m)=e^{-|m|^2 t}z^m\pl.\]
We then proceed as in the previous example to compute the CB-return time:
\begin{align*}
\norm{T_t-E_\tau:L_1(\mathbb{T}^d)\to L_\infty(\mathbb{T}^d)}{}&=\norm{\sum_{m\in \Z^d, m\neq 0}e^{-m^2t}z^mw^{-m}}{L_\infty(\mathbb T^d \times \mathbb T^d)}\\
&=\norm{\sum_{m\in \Z^d, m\neq 0}e^{-m^2t}z^m}{L_\infty(\mathbb T^d)} \\
&= \norm{k_t^{(d)}-1}{L_\infty(\mathbb T^d)}\\
&= k_t^{(d)}(e)-1\\
&= (2\sum_{m=1}e^{-m^2t})^d=f(t)^d.
\end{align*}
where $f(t):=2\sum_{m=1}e^{-m^2t}$ is as in the previous example.
Thus we have a CB-return time estimate depending on the dimension $d$
\[t_{cb}(d)= \inf\{t\pl |\pl  f(t)\le 2^{-\frac{1}{d}}\}\pl.\]
Using the same heat kernel estimates as in the previous example, we then conclude that
\[
(1+\frac{1}{d})\ln 2 \le t_{cb}(d) \le \ln (2^{(1+\frac{1}{d})}+1).
\]
For example,
$t_{cb}(2)\le 1.35$ and $t_{cb}(3)\le 1.26$. (Numerical suggests $t_{cb}(2)\le 1.08$ and $t_{cb}(3)\le 0.98$). }
\end{exam}
Note that the CLSI constant $(4t_{cb})^{-1}$ obtained from the above approach is monotone increasing for $d$, which is better than tenzorisation.
This leads to the following dimension free estimates.
\begin{theorem}Let $d\ge 1$ and $\mathbb{T}^d$ be the unit $d$-torus. The heat semigroup on $\mathbb{T}^d$ (in the above normalization) satisfies complete Fisher monotonicity and $\displaystyle\la$-CLSI for $\la=(4\ln 3)^{-1}$.
\end{theorem}
\begin{proof}Denote $T_{\T^d,t}$ as the semigroup on $\T^d$. Denote CLSI$(T_{\T^d,t})$ as the optimal CLSI constant of $T_{\T^d,t}$. Then by example \ref{dtorus}, we have
\[\text{CLSI}(T_{\T^d,t})\ge \Big(4\inf\{\pl t\pl |\pl (2\sum_{m=1}e^{-m^2t})^d\le 1/2\pl\}\Big)^{-1}\ge \Big(4\ln (2^{(1+\frac{1}{d})}+1) \Big)^{-1} \]
For $m\le d$, $\T^d=\T^m\times \T^{d-m}$. Consider the embedding $\pi_{m,d}:C(\T^m)\to C(\T^d)$
\[ \pi_{m,d}(f)=f\ten 1_{d-m}\pl, f\in C(\T^m)\pl.\]
where $1_{d-m}$ is the identity function on $\T^{d-m}$. Namely, $\pi_{m,d}(f)(z_1,\cdots, z_d)=f(z_1,\cdots,z_m)$. It is clear that
\[\pi_{m,d}\circ T_{\T^m,t}=T_{\T^d,t}\circ  \pi_{m,d}\pl. \]
Hence the heat semigroup $T_{\T^m,t}$ on $m$-torus is a sub-semigroup for $T_{\T^d,t}$ on $d$-torus. We have for any $d\ge m$,
\[   \text{CLSI}(T_{\T^m,t})\ge \text{CLSI}(T_{\T^d,t})\ge \Big(4\ln (2^{(1+\frac{1}{d})}+1) \Big)^{-1}\pl. \]
Taking $d\to\infty$, we have $\text{CLSI}(T_{\T^m,t})\ge (4\ln 3)^{-1}$ for any $m$. That completes the proof.
\end{proof}

\begin{rem}{\rm It was proved by Weissler \cite{cirle} that on the circle $\mathbb{T}$, both the heat semigroup $T_t(z^m)=e^{-m^2}z^m$ and the Possion semigroup $P_t(z^m)=e^{-|m|t}z^m$ satisfies sharp $1$-LSI hence sharp $1$-MLSI (because spectral gap is $1$). We will show in the second part of this series that the Possion semigroup $P_t$ on $\mathbb{T}$ satisfies sharp $\ARic\ge 1$ and hence sharp $1$-CLSI. }
\end{rem}

\begin{exam}[Finite Groups]{ \label{finitegroup}\rm Let $G$ be a finite group and $l_\infty(G)$ be the function space on $G$ equipped with counting probability measure. Let
\[T_t:l_\infty(G)\to l_\infty(G)\pl, (T_tf)(g)=\sum_{g\in G} k_t(g^{-1}h)f(h)\]
be a symmetric central Markov semigroup with kernel function $k_t\in Z l_1(G)$. Let $A$ be generator of $T_t$, which acts on the  $l_2(G)$:
\[A:l_2(G)\to l_2(G), A(e_h)=\sum_{g \in G}A_{g,h}e_g\]
The entries of $A$  are given by
\[ A_{g,h}=\begin{cases}
              \sum_{h\neq g}w_{g,h}, & \mbox{if } h=g \\
             -w_{g,h}, & \mbox{otherwise}.
           \end{cases}\]
where $w_{g,h}>0$ are the transition rates. If $T_t$ is symmetric and central, \[w_{g,h}=w_{h,g}=w_{sg, sh} = w_{gs,hs}\pl,\forall \pl s,g,h\in G\pl.\]
Here we use the derivation of finite Markov chain from \cite{ME11}. Denote $B=\sum_{g\neq h}\sqrt{w_{g,h}}e_{g,h}$,
where $e_{g,h}$ are matrix units in $B(l_2(G))$.
Consider the standard embedding $\pi: l_\infty(G)\hookrightarrow  B(l_2(G))$ as diagonal matrices $\pi(f)=\sum_{g}f(g)e_{g,g}$.
We have the following derivation.
\[\delta:l_\infty(G)\to B(l_2(G))\pl, \pl \delta(f)=\sum_{g,h}b_{g,h}(f(h)-f(g))e_{g,h}=i[B,\pi(f)]\] For the gradient form,
\begin{align*}2\Gamma(e_g,e_h)&=\Big(e_g^*(Ae_h)+(Ae_g)^*e_h-A(e_ge_h)\Big)
\\&=\begin{cases}
     \sum_{s\neq g}w_{s,g}(e_s+e_g), & \mbox{if } g=h \\
      -w_{g,h}(e_g+e_h), & \mbox{otherwise}.
    \end{cases}
\end{align*}
Note that $[B,\pi(e_{g})]=\sum_{s\neq g}\sqrt{w_{s,g}}(e_{s,g}-e_{g,s}) $. Then for $g\neq h$
\begin{align*}E([B,\pi(e_{g})]^*[B,\pi(e_{h})])&=
E\Big((\sum_{s\neq g}\sqrt{w_{s,g}}(e_{g,s}-e_{s,g}))(\sum_{r\neq h}\sqrt{w_{r,h}}(e_{r,h}-e_{h,r}))\Big)
\\&=E\Big( \sum_{r}\sqrt{w_{r,h}}\sqrt{w_{r,g}}e_{g,h}-\sum_{s}\sqrt{w_{s,g}}\sqrt{w_{g,h}}e_{s,h}
\\&-\sum_{r}\sqrt{w_{h,g}}\sqrt{w_{r,h}}e_{g,r}+ \delta_{g,h}\sum_{s,r}\sqrt{w_{s,g}}\sqrt{w_{r,g}}e_{s,r} \Big)
\\&=-w_{g,h}e_{h}
-w_{h,g}e_{g}
\end{align*}
For $g=h$,
\begin{align*}E([B,\pi(e_{g})]^*[B,\pi(e_{h})])&=
E\Big((\sum_{s\neq g}\sqrt{w_{s,g}}(e_{g,s}-e_{s,g}))(\sum_{r\neq h}\sqrt{w_{r,h}}(e_{r,h}-e_{h,r}))\Big)
\\&=\sum_{r\neq g}w_{r,g}e_{g}+\sum_{s\neq g}w_{s,g}e_{s}
\\&= \sum_{s}w_{s,g}(e_{s}+e_g)
\end{align*}
Thus we have verified that
\[\Gamma(e_g,e_h)=E(\delta(e_g)^*\delta(e_h)) \pl.\]
which extends bi-linearly to $l_\infty(G)\times l_\infty(G)$.
Now we have
\begin{align*} \delta\circ T_t(e_g)&=\delta(\sum_{r}k_t(r)e_{gr^{-1}})
\\&=\sum_{r}k_t(r)\sum_{s\neq gr^{-1}}\sqrt{w_{s,gr^{-1}}}(e_{s,gr^{-1}}-e_{gr^{-1},s})
\\&=\sum_{r}k_t(r) U_{r}^*\Big(\sum_{sr\neq g} \sqrt{w_{sr,g}}(e_{sr,g}-e_{g,sr})\Big)U_{r}
\\&=\sum_{r}k_t(r) U_r^*\Big(\delta(e_g)\Big)U_r
\\&=\hat{T}_t \circ\delta(e_g)
\end{align*}In the third equality above we used the central property $w_{s,gr^{-1}}=w_{sr,g}$. The extension semigroup on $B(l_2(G))$ is \[\hat{T}_t(\rho)=  \sum_{r}k_t(r) U_r\rho U_r^*\pl,\]
where $U_re_g=U_re_{gr}$ is the right shifting unitary. $\hat{T}_t$ is a extension of $T_t:l_\infty(G)\to l_\infty(G)$ on $ B(l_2(G))$. Indeed,
\[T_t(e_{g,g})=\sum_{r}k_t(r)e_{gr^{-1}, gr^{-1}}=\sum_{r}k_t(r)U_{r^{-1}}e_{g,g} U_{r^{-1}}^*=\sum_{r}k_t(r)U_r^*e_{g,g} U_r\pl.\]
This verifies that $T_t$ satisfies $0$-$\ARic$ via a construction different from Lemma \ref{group}.}
\end{exam}



\begin{cor} Let $T_t:l_\infty(G)\to l_\infty(G)$ be a central Markov semigroup with spectral gap $\si$. Then $T_t$ satisfies $\ARic\ge 0$, complete Fisher monotonicity and $\la$-CLSI for \[\la=\frac{\si}{4(\log 2|G|)}
\pl.\]
\end{cor}
\begin{proof}This follows from Theorem \ref{CLSI3}, Proposition \ref{fd} and $D_{cb}(l_\infty(G)||\mathbb{C})=|G|$.
\end{proof}

\subsection{Generalized Depolarizing Semigroups} \label{depolar}  Let $\N\subset \M$ be a subalgebra and let $E:\M\to \N$ be the conditional expectation. We now discuss curvature bounds and MLSI (resp. CLSI) constants for the generalized depolarizing semigroup
\[T_t(\rho)=e^{-\la t}\rho+(1-e^{-\la t})E(
\rho)\pl.\]
The generator is $A=\la(I-E)$ whose spectral gap is clearly $\la$ (here $I$ is the identity operator on $L_2(\M)$). In the following we show that $T_t$ has $ \la/2$-GE. This result is independently obtained by Melchior Wirth and Haonan Zhang and the case for ergodic depolarizing semigroup on matrix algebras was obtained in \cite[Section 3.4]{DR}.

\begin{theorem}\label{depo}The generalizing depolarizing semigroup
\[T_t(\rho)=e^{-\la t}\rho+(1-e^{-\la t})E(
\rho)\pl.\]
satisfies $(\la/2)$-GE.
\end{theorem}

\begin{proof}
Let $(\A,\hat{\M},\delta)$ be a derivation triple of $T_t$. Since $\delta(x)=0$ for $x\in \N$, we have for $x\in \A$,
\[ \delta(T_t(x))=\delta\big(e^{-\la t}(x-E(x))+E(
x)\big)=e^{-\la t}\delta(x)\pl.\]
Then we have
\begin{align*}\norm{\delta(T_t(x))}{\rho}^2=
\norm{e^{-\la t}\delta(x)}{\rho}^2=e^{-2\la t}\norm{\delta(x)}{\rho}^2,
\end{align*}
where
\[\norm{\delta(x)}{\rho}^2=\int_{0}^{1}\tau(\delta(x)^*\rho^{s}\delta(x)\rho^{1-s}) ds.\]
It follows from Lieb's concavity theorem \cite{Lieb73} that for each $s\in [0,1]$,
\[(\rho,\si)\to \tau(\delta(x)^*\rho^{s}\delta(x)\si^{1-s})\]
is jointly concave for $(\rho,\si)$. For $\rho_t:=T_t(\rho)=e^{-\la t}+(1-e^{\la t})E(\rho)$,
\[\tau(\delta(x)^*\rho_t^{s}\delta(x)\rho_t^{1-s})\ge e^{-\la t}\tau(\delta(x)^*\rho^{s}\delta(x)\rho^{1-s})+(1-e^{\la t})\tau(\delta(x)^*E(\rho)^{s}\delta(x)E(\rho)^{1-s})\]
Integrating over $s$,
\[ \norm{\delta(x)}{T_t(\rho)}^2\ge e^{-\la t} \norm{\delta(x)}{\rho}^2+(1-e^{\la t})\norm{\delta(x)}{E(\rho)}\ge e^{-\la t} \norm{\delta(x)}{\rho}^2\pl.\]
Then
\begin{align*}\norm{\delta(T_t(x))}{\rho}^2=e^{-2\la t}\norm{\delta(x)}{\rho}^2\le e^{-2\la t}e^{\la t}\norm{\delta(x)}{T_t(\rho)}=e^{-\la t}\norm{\delta(x)}{T_t(\rho)}
\end{align*}
which proves the gradient estimates.
\end{proof}
\begin{rem}{\rm In an upcoming paper, we will prove a stronger result that $T_t=e^{-\la(I-E)t}$ satisfies $\ARic\ge \la/2$ based the free product property discussed there.
}\end{rem}

Note that the above theorem implies the generator $A=(I-E)$ has $1/2$-CLSI. This can be verified directly via
its Fisher information
\begin{align}I(\rho)=&\tau((I-E)(\rho)\log \rho )=\tau(\rho\log \rho)-\tau(E(\rho)\log \rho )\nonumber\\
=&\tau(\rho\log \rho-\rho\log E(\rho))+ \tau(\rho\log E(\rho)-E(\rho)\log \rho)\nonumber\\
=&\tau(\rho\log \rho-\rho\log E(\rho))+ \tau(E(\rho)\log E(\rho)-E(\rho)\log \rho)\nonumber\\
=&D(\rho|| E(\rho))+D(E(\rho)|| \rho)\ge D(\rho|| E(\rho)) \label{11}
\end{align}
where in the third equality we used the definition of the conditional expectation.
It follows from $D(E(\rho)|| \rho)\ge 0$ that $A=(I-E)$ has $1/2$-MLSI and also $1/2$-CLSI by the same argument for $(I-E)\ten \id$. In the following discussion, we denote $MSLI(A)$ (resp. $CLSI(A)$ and $GE(A)$) as the optimal constant $\la$ of MSLI (resp. CLSI and GE) for the generator $A$.

\begin{exam}[Depolarizing Semigroup]\label{depolarize}{\rm Let $M_d$ be the algebra of $d\times d$ matrix.
Consider the depolarizing semigroup
\begin{align*} &D_t:M_d\to M_d\pl, \pl D_t(\rho)=e^{- t}\rho +(1-e^{-t})\tau_d(\rho) 1,
\end{align*}where $\tau_d(x)=\frac{1}{d}\text{Tr}(x)$ is the normalized matrix trace on $M_d$. It is proved in \cite{KT13} that the optimal LSI constant is \[LSI(I-\tau_d)= \frac{2-4/d}{\log(d-1)}\pl, \pl LSI(I-\tau_2)=1\]
This implies
\[MLSI(I-\tau_d)\ge \frac{2-4/d}{\log(d-1)}\pl,\]
For curvature bounds,
Melchior and Zhang proves that $GE(I-\tau_d)\ge \frac{1}{2}+\frac{1}{2d}$. Here we show that $\text{GE}(I-\tau_3)\le \text{MLSI}(I-\tau_3)<1$. In $M_3$, we choose the normalized density $\rho=\frac{3}{2}e_1+\frac{3}{4}e_2+\frac{3}{4}e_3$ where $e_1,e_2,e_3$ are orthogonal rank one projections. Then
\begin{align*} & D(\rho|| 1)=\frac{1}{2}\log (3/2)+\frac14\log (3/4)+\frac14\log (3/4)=\frac12\log (9/8)=\log (3/2\sqrt{2})
\\ & D(1|| \rho)=\frac{1}{3}\log (2/3)+\frac{1}{3}\log (4/3)+\frac{1}{3}\log (4/3)=\frac{1}{3}\log (32/27)=\log (2^{5/3}/3)<\log (3/2\sqrt{2})
\end{align*}
This means $D(\rho|| 1)> D(1|| \rho)$  and
\[ I(\rho)=D(\rho|| 1)+ D(1|| \rho)<2D(\rho|| 1)\]
This implies on $M_3$, the depolarizing semigroup $A=I-\tau_3$ does not have $1$-MLSI nor $1$-GE. Similar examples can be found for other $d\ge 3$.
}\end{exam}

\begin{rem}{\rm
By \eqref{11}, the optimal MLSI constant is
\[\text{MLSI}(I-\tau_d)=\frac{1}{2}(1+\inf_{\rho\in S(M_d)} \frac{D( 1||\rho)}{D( \rho||1)})\pl.\]
It is clear that $(I-\tau_d)$ has the same MLSI constant for the classical depolarizing semigroup
\[S_t:l_\infty^d\to l_\infty^d\pl, \pl S_t(f)=e^{- t}f +(1-e^{-t})\frac{(\sum_{i}f(i))}{d} 1\pl.\]
Maas and Erbar showed in \cite{ME11} that $GE(S_t)\ge \frac{1}{2}+\frac{1}{2d}$.}
\end{rem}
 We show that the above GE constant also holds for $M_n$. We are indebt to Melchior Wirth for pointing out our earlier mistake on the following proposition.
\begin{prop}\label{md}The $d$-dimensional depolarizing semigroup
 \[ D_t:M_d\to M_d\pl, \pl D_t(\rho)=e^{- t}\rho +(1-e^{- t})\tau_d(\rho) 1\]
 satisfies $\frac{1}{2}+\frac{1}{2d}$-GE.
 \end{prop}
 \begin{proof}
 Note that for any derivation $\delta$ of $T_t$,
\[\delta(a)=\delta(a-\tau_d(a))\pl, \delta(T_t(a))=e^{-t}(a-\tau_d(a))\pl.\]
Then
\begin{align*}\norm{\delta(T_t(a))}{L_2(\hat{\M},\rho)}^2=e^{-2 t}\norm{\delta(a)}{L_2(\hat{\M},\rho)}^2
\end{align*}
Let $\al>0$. The $D_t$ satisfies $ \frac{1}{2}+\frac{1}{2\al}$-GE means that for any $a\in M_d$,
\begin{align*}e^{-2t}\norm{\delta(a)}{L_2(\hat{\M},\rho)}^2=\norm{\delta(T_t(a))}{L_2(\hat{\M},\rho)}^2\le e^{-2(\frac{1}{2}+\frac{1}{2\al}) t}\norm{\delta(a)}{L_2(\hat{\M},T_t(\rho))}^2
\end{align*}
This is equivalent to the function
\[h(t):=e^{(1-\frac{1}{\al})t}\norm{\delta(a)}{L_2(\hat{\M},T_t(\rho))}^2\]
is increasing.
Denote the function $f(t):=\norm{\delta(a)}{L_2(\hat{\M},T_t(\rho))}^2$. Write $\rho_t=T_t(\rho)$ for a density $\rho$. We have the derivative,  \[\frac{d}{dt}\rho_t=-(I-\tau)(\rho)= (1-\rho)\] and ($\hat{\tau}$ is the trace on the derivation triple)
\begin{align*}
\frac{df(t)}{d t}|_{t=0}&= \frac{d}{d t}|_{t=0}\Big( \int_0^1 \hat{\tau}(\delta(a)^*\rho_t^{1-s} \delta(a) \rho_t^s)ds\Big)\\
&=  \int_0^1 \Big(\frac{d}{d t}|_{t=0} \hat{\tau}(\delta(a)^*\rho_t^{1-s} \delta(a) \rho_t^s)\Big)ds
\end{align*}
Let $\rho=\sum_{j}p_je_j$ be the orthogonal decomposition of $\rho$. By double operator integral,
\begin{align*} \frac{d}{dt}\rho^s_t|_{t=0}= & \sum_{j,k}\frac{p_j^s-p_k^s}{p_j-p_k}e_j(1-\rho)e_k \\ =&(1-\rho)\sum_{j,k}\frac{p_j^s-p_k^s}{p_j-p_k}e_je_k
\\ =&\la s (1-\rho)\rho^{s-1}
\end{align*}
and similarly $\frac{d}{dt}\rho^{1-s}_t|_{t=0}=\la (1-s) (1-\rho)\rho^{-s}$.
For a bi-viariable function $F:(0,\infty)\times (0,\infty) \to (0,\infty)$, we introduce the notation
\[I_{F,\rho}(X)=\sum_{j,k}F(p_j,p_k)e_j X e_k\pl,\]
Then
\begin{align*}
& \frac{d}{dt}(\rho_t^{1-s} \delta(a) \rho_t^s)|_{t=0}
 = (1-\rho)I_{F_s,\rho}(\delta(a))+I_{G_s,\rho}(\delta(a))(1-\rho)
\end{align*}
where $F_s(x,y)=(1-s)(1-x)x^{-s}y^{s}$ and $G_s=s x^{1-s}y^{s-1}(1-y)$. Integrating over $s$ on $[0,1]$, we have
\begin{align*}
 \frac{d f(t)}{d t}|_{t=0}=& \frac{d}{d t}|_{t=0}\Big( \int_0^1 \hat{\tau}(\delta(x)^*\rho_t^{1-s} \delta(x) \rho_t^s)ds\Big)
 \\ = &\int_0^1\tau(\delta(x)^*I_{F_s,\rho}(\delta(x)) )+ \tau(\delta(x)^*I_{G_s,\rho}(\delta(x)))ds
 \\ = &\tau(\delta(x)^*I_{H,\rho}(\delta(x)))
\end{align*}
where $H$ is the function given by
\begin{align*}
H(x,y)=\int_0^1 F_{s}(x,y)+G_s(x,y)ds&=\frac{(x - y) (x - y - x y (\log(x) - \log(y)))}{x y (\log(x) - \log(y))^2}\\
&=\frac{(x - y)}{\log(x) - \log(y)} \frac{ (x - y - x y (\log(x) - \log(y)))}{x y (
\log(x) - \log(y))}
\\
&=\frac{(x - y)}{\log(x) - \log(y)} \Big(\frac{ (x - y)}{x y (\log(x) - \log(y))} - 1\Big)
\end{align*}
here $\log$ is natural log.
On the other hand,
\begin{align*}f(0)=&\int_0^1\tau(\delta(a)^*\rho^{1-s}\delta(a)\rho^{s})\pl,
\\=&\tau(\delta(a)^*I_{J,\rho}\delta(a))
\end{align*}
where $\displaystyle J(x,y)=\frac{x-y}{\log x-\log y }$.
Then the derivative of $h(t)=e^{(1-\frac{1}{\al})t}f(t)$ is
\begin{align*}
h'(0)=&(1-\frac{1}{\al})f(0)+f'(0)
\\=& (1-\frac{1}{\al})\tau(\delta(x)^*I_{J,\rho}(\delta(x)))+\tau(\delta(x)^*I_{H,\rho}(\delta(x)))
\end{align*}
Thus it suffices to require $(1-\frac{1}{\al})J+H$ is a positive function on the spectrum of $\rho$. Indeed,
\begin{align*} (1-\frac{1}{\al})J(x,y)+H(x,y)
&= (1-\frac{1}{\al})\frac{(x - y)}{\log(x) - \log(y)} + \frac{(x - y)}{\log x - \log y} \Big(\frac{ (x - y)}{x y (\log x - \log y)} - 1\Big)
\\&= \frac{(x - y)}{\log x - \log y} \Big(\frac{ (x - y)}{x y (\log x - \log y)} -\frac{1}{\al}\Big)
\end{align*}
Because $\frac{(x - y)}{\log x - \log y}\ge 0$, it suffices to require
\[\frac{ (x - y)}{x y (\log x - \log y)} - \frac{1}{\al}>0\]
or equivalently
\[ \frac{x y (\log x - \log y)}{x - y}\le \al\pl.\]
Here for $M_d$, the domain of $(x,y)$ is contained in $S_d:=\{0\le x,y\le d \pl \}$ since $x,y$ are eigenvalues of a normalized density $\rho\in M_d$. By elementary calculus, one can show
 \[ \max_{(x,y)\in S_d}\frac{x y (\log x - \log y)}{x - y}=d\pl.\]
 Thus $\al$ can be $d$ and we finishes the proof.
 \end{proof}

We now use a similar idea to consider the MLSI constant of
$D_t\ten \id_{2}:M_2\ten M_2\to M_2\ten M_2$ where $D_t$ is the depolarizing on $M_2$. Let $E:M_2\ten M_2\to M_2\ten M_2 \pl, E(\rho)=\big(\tau\ten \id(\rho)\big)\ten 1$ be the partial trace map. Consider the basis of Bell states
\begin{align*}
&\ket{\phi_1}=\frac{1}{\sqrt{2}}(\ket{0}\ket{0}+\ket{1}\ket{1})
\pl,\pl \ket{\phi_2}=\frac{1}{\sqrt{2}}(\ket{0}\ket{0}-\ket{1}\ket{1})
\\ &\ket{\phi_3}=\frac{1}{\sqrt{2}}(\ket{0}\ket{1}+\ket{0}\ket{1})
\pl,\pl \ket{\phi_4}=\frac{1}{\sqrt{2}}(\ket{0}\ket{1}-\ket{0}\ket{1})
\end{align*}
Using the identification
\[\ket{0}\ket{0}\to \ket{1}\pl, \ket{0}\ket{1}\to \ket{2}\pl,\ket{1}\ket{0}\to \ket{3}\pl,\ket{1}\ket{1}\to \ket{4} \]
we have the densities in $M_4\cong M_2\ten M_2$ represented as
\begin{align*}
&\phi_1=\left[\begin{array}{cccc}\frac{1}{2}&0&0&\frac{1}{2}\\ 0&0&0&0\\
0&0&0&0\\ \frac{1}{2}&0&0&\frac{1}{2}
\end{array}\right]
&\phi_2=\left[\begin{array}{cccc}\frac{1}{2}&0&0&-\frac{1}{2}\\ 0&0&0&0\\
0&0&0&0\\ -\frac{1}{2}&0&0&\frac{1}{2}
\end{array}\right]\\
&\phi_3=\left[\begin{array}{cccc} 0&0&0&0\\ 0&\frac{1}{2}&\frac{1}{2}&0\\ 0&\frac{1}{2}&\frac{1}{2}&0\\
0&0&0&0
\end{array}\right]
&\phi_4=\left[\begin{array}{cccc} 0&0&0&0\\ 0&\frac{1}{2}&-\frac{1}{2}&0\\ 0&-\frac{1}{2}&\frac{1}{2}&0\\
0&0&0&0
\end{array}\right]
\end{align*}
Now we choose the state $\rho=\frac{5}{8}\phi_1 +\frac{1}{8}(\phi_2+\phi_3+\phi_4)$. The reduced density is
\[  E(\rho)=\left[\begin{array}{cc} \frac{1}{2}&0\\ 0&\frac{1}{2}
\end{array}\right]\ten \frac{1}{2}=\frac{1}{4}1\ten 1\pl.\]
Thus
\begin{align*} &D(\rho||E(\rho))=D(\rho|| \frac{1}{4})=\frac{5}{8}\log (5/2)+\frac{3}{8}\log (1/2)\simeq 0.313\\
&D(E(\rho)||\rho)=D(\frac{1}{4}|| \rho)=\frac{1}{4}\log (2/5)+\frac{3}{4}\log 2\simeq 0.291
\end{align*}
Then we have $D(\rho||E(\rho))>D(E(\rho)||\rho)$, which implies that $\text{GE}((I-\tau_2) \ten \id_{M_2})\le \text{MLSI}((I-\tau_2) \ten \id_{M_2})<1$. Note that $ \text{MLSI}(I-\tau_2)=1$.
We have the following corollary.

\begin{prop}
Let $D_t$ be the depolarizing semigroup on $M_2$ and $(I-\tau_2)$ be its generator. Then
\begin{align*}
\text{MLSI}((I-\tau_2)\ten \id_{M_2})<1=\text{MLSI}(I-\tau_2)
\end{align*}
In particular, $CLSI(I-\tau_2)<\text{MLSI}(I-\tau_2)$.
\end{prop}
For classical Markov semigroups, the MLSI is stable under tensorisation. The above example shows that tensorisation of MLSI does not holds for quantum cases if we allow non-ergodic semigroup.


\subsection{Schur multipliers} \label{schur}
Let $M_m$ be the $m\times m$ matrix algebra and $a=(a_{ij})_{i,j=1}^{m}\in M_m$. The Schur multiplier of $a$ is
\[T_a:M_n\to M_n \pl, \pl T_a(x_{ij})=(a_{ij}x_{ij})\]
Consider a semigroup of Schur multiplier $T_t:M_n\to M_n\pl, \pl T_t((x_{ij}))=(e^{-b_{ij}t}x_{ij})$. The generator is the Schur multiplier of $b=(b_{ij})$,
\[A((x_{ij}))=(b_{ij}x_{ij})\pl.\]
By Schoenberg's theorem \cite{Schoenberg38}, $T_t$ is a symmetric quantum Markov semigroup (unital completely positive and self-adjoint) if and only if $b_{ii}=0, b_{ij}=b_{ji}\ge 0$ and \emph{conditionally negative definite}, i.e. for any real sequence $(c_1,\cdots, c_m)$ with $\sum_{i=1}^mc_i=0$,
\[\sum_{i,j=1}^mc_ic_jb_{ij}\le 0\pl.\]
Moreover, there exists a real Hilbert space $H$ and a family of vector $b(1), \cdots, b(n)\in H$ such that
\[b_{ij}=\norm{b(i)-b(j)}{}^2\pl.\]
For $T_t$, the fixed point subalgebra $\N$ is
\[\N=\{(x_{ij})\in M_n\pl | \pl x_{ij}=0 \pl \text{for all $(i,j)$ that $b_{ij}\neq 0$} \}\pl,\]where $e_{ij}\in M_m$ are the matrix units. It is clear that the diagonal matrices $l_\infty^m\subset \N$. Thus $T_t$ are always non-ergodic.
Because $e_{ij}$ are eigenvectors of the generator $A$ with eigenvalue $b_{ij}$, the spectral gap is
\[\si=\min \{\pl b_{ij}\pl  |\pl b_{ij}\neq 0\pl \}\pl.\]
The gradient form is given by
\begin{align*} \Gamma(e_{ij},e_{lk})&=\frac{1}{2}\delta_{il}(b_{ij}+b_{lk}-b_{jk})e_{jk}\pl.
\end{align*}Here $\delta_{il}$ is the Kroenecker-delta notation.
For $i=l$, we have
\begin{align*}
\Gamma(e_{ij},e_{ik})&=\frac{1}{2}(\norm{b(i)-b(j)}{}^2+\norm{b(i)-b(k)}{}^2-\norm{b(j)-b(k)}{}^2)e_{jk}
\\ &=\frac{1}{2}(\norm{b(i)-b(j)}{}^2+\norm{b(i)-b(k)}{}^2-\norm{b(j)-b(k)}{}^2)e_{jk}
\\ &=\lan b(i)-b(j),b(i)-b(k)\ran e_{jk}
\end{align*}
Recall that for a real Hilbert space $H$, an $H$-isonormal process on a standard probability
space $(\Omega, m)$ is a linear mapping $W: H \to L_0(\Omega)$ satisfying the following properties:
\begin{enumerate}
\item[i)] for any $v\in H$, the random variable $W(v)$ is a centered real Gaussian.
\item[ii)] for any $v_1,v_2\in H$, we have $E_\Omega(W(v_1)W(v_2))=\lan v_1,v_2\ran_H $
\item[iii)] The linear span of the products $\{W(v_1)W(v_2)\cdots W(v_n)\pl |\pl v_1,\cdots, v_n\in H\}$ is dense in the real Hilbert space $L_2(\Omega)$
\end{enumerate}
Here $L_0(\Omega)$ denote the space of measurable functions on $\Omega$. 
Now we define the derivation
\[ \delta: M_m\to M_m\ten L_2(\Omega)\pl, \delta(e_{ij})=e_{ij}\ten \sqrt{-1}(W(b(i))-W(b(j)))\pl.\]
We verify that $\delta$ is a derivation,
\begin{align*}
\delta(e_{ij})e_{jk}+e_{ij}\delta(e_{jk})&=e_{ik} \ten \sqrt{-1}(W(b(i))-W(b(j)))+e_{ik}\ten  \sqrt{-1}(W(b(j))-W(b(k)))\\ &=e_{ik} \ten \sqrt{-1}(W(b(i))-W(b(k)))=\delta(e_{ik})=\delta(e_{ij}e_{jk}).
\end{align*}
Moreover for the gradient form
\begin{align*} E(\delta(e_{ij})^*\delta(e_{lk}))
&=E\Big(  \big(e_{ji}\ten (W(b(i))-W(b(j)))\big) \big(e_{lk}\ten (W(b(l))-W(b(k)))\big)\Big)
\\ &=\delta_{il} e_{jk}\ten E\Big( (W(b(i))-W(b(j)))(W(b(i))-W(b(k)))\Big)
\\ &=\delta_{il} \lan b(i)-b(j),b(i)-b(k)\ran e_{jk}\pl.
\end{align*}
Then it is readily seen that
\[\delta\circ T_t=(T_t\ten id_{\Omega}) \circ \delta\pl,\]
where $T_t\ten id_{\Omega}$ is the extension of $T_t$ on $M_m\ten L_\infty(\Omega)$. By Theorem \ref{alg}, this implies $T_t$ satisfies $0$-GRic. Combined with CB-return time estimates in Proposition \ref{fd}, we have
\begin{theorem}
Let $T_t:M_m\to M_m, T_t((x_{ij}))=(e^{-b_{ij}t}x_{ij})$ be a symmetric quantum Markov semigroup of Schur multipliers. Then $T_t$ satisfies $0$-$\ARic$ and complete Fisher monotonicity. Denote $\si=\min\{ b_{ij}\pl |\pl b_{ij}\neq 0\}$ as the spectral gap of $T_t$. Then $T_t$ satisfies $\la$-CLSI with constant
\[ \la=\frac{\si}{4(D_{cb}(M_m||\N)+\log 2)}\]
In particular, $D_{cb}(M_m||\N)\le D_{cb}(M_m||l_\infty^m)=\log m$.
\end{theorem}

\subsection{Random unitary channels} \label{randomu}
A CPTP map $T:M_m\to M_m$ is called a {\it random unitary channel} if it is a convex combination of unitary conjugations,
\[T(\rho)=\sum_{j=1}^np_j U_j \rho U_j^* \quad (p_i \ge 0, \ \sum_i p_i =1)\pl .\]
In this subsection, we discuss semigroups of random unitary channels arising from group representations. Let $G$ be a finite group. Recall that a projective unitary representation $U: G\to U(M_m)$ satisfies
\[U_gU_{h}=\si(g,h)U_{gh}\pl, \pl \forall \pl g,h\in G\]
where $\si:G\times G\to \mathbb{C}$ is a group $2$-cocycle with $|\si(g,h)|=1$.
Let $T_t:M_m\to M_m$ be the quantum Markov semigroup given by
\[ T_t(\rho)=\frac{1}{|G|}\sum_{g}k_t(g)U_g \rho U_g^*\pl.\]
where $k_t(g)$ is the weight function that satisfies $k_t(g)\ge 0$, $\sum_{g}k_t(g)=1$ and
 \[k_{t+s}(g)=\frac{1}{|G|}\sum_{h}k_t(gh^{-1})k_s(h) = (k_t\star k_s)(g)\pl.\]
Thus $k_{t}$ forms the right invariant kernel on $G$. Let \[S_t:l_\infty(G)\to l_\infty(G), S_t(f)(g)=\sum_{h}k_t(gh^{-1})f(h)\pl.\] be the right invariant Markov semigroup on $l_\infty(G)$.
We have the transference
\begin{equation}
 \begin{array}{ccc}  l_{\infty}(G, M_m)\pl\pl &\overset{ S_t\ten \id_{M_m}}{\longrightarrow} & l_{\infty}(G, M_m) \\
                    \uparrow \al    & & \uparrow \al  \\
                    M_m\pl\pl &\overset{T_t}{\longrightarrow} & M_m
                     \end{array} \pl .
                     \end{equation}
where $\al:M_m\to l_\infty(G, M_m), \al(x)(g)=U_g xU_g^*$ is a trace preserving $*$-monomorphism. Thus $T_t=(S_t\ten \id)|_{\al(M_m)}$ is a subsystem of the semigroup $(S_t\ten \id_{M_m})$.
\begin{theorem}Let $G$ be a finite group and let $U:G\to M_m$ be a projective unitary representation. Let $T_t:M_m\to M_m$ be the a quantum Markov semigroup given by
\[ T_t(\rho)=\frac{1}{|G|}\sum_{g}k_t(g)U_g \rho U_g^*\pl.\]
Suppose $k_t$ is central and $T_t$ has spectral gap $\si$.
Then $T_t$ satisfies complete Fisher monotonicity and $\la$-CLSI with constant
\[ \la=\frac{\si}{4(\log 2m^2)}\]
\end{theorem}
\begin{proof}

If $k_t$ are central, it follows from Theorem \ref{group} that the classical semigroup $S_t$ satisfies complete Fisher monotonicity. which pass to $T_t$ as a subsystem. The CLSI constant follows from Proposition \ref{fd} and $D_{cb}(M_m||\N)\le D_{cb}(M_m||\mathbb{C})=m^2$.
\end{proof}

\begin{exam}{\rm Recall the $m$-dimensional generalized Pauli matrices are
\[X\ket{j}=\ket{j+1}\pl, Z\ket{j}=e^{\frac{2\pi i j}{m}}\ket{j}\pl.\]
It is clear that $\{X^kZ^l\}$ forms a projective representation of $\mathbb{Z}^2_m$. Since $\mathbb{Z}^2_m$ is abelian, so every function on $\mathbb{Z}^2$ is a central. Thus the above theorem applies to every semigroup of random Pauli unitaries
\[T_t(\rho)=\frac{1}{m^2}\sum_{j,l}k_t(j,l)X^jZ^l\rho (X^jZ^l)^*\pl.\]

}
\end{exam}

\appendix
\section{}
In this appendix we provide the approximation lemmas in terms of entropy. We start with a standard density argument.

\begin{lemma}\label{densestate}Suppose $\A\subset \M$ is a $w^*$-dense unital $\ast$-subalgebra $\A\subset \M$. Denote $B$ as the unit ball of $\M$. Then $\A$ is norm dense in $L_2(\M)$ and $L_1(\M)$. Moreover, the positive part $\A_+$ (resp. $\A_+\cap B$) is dense in $L_1(\M)_+$ (resp. $L_1(\M)_+\cap B$).
  \end{lemma}
  \begin{proof}By Kaplansky density theorem (c.f. \cite[Theorem 4.8]{takesaki}), $\A\cap B$ is also strong operator topology (SOT) dense in $\M\cap B$. Then for any $\xi\in L_2(\M)$, we have a net $(x_\al)\subset \A$ such that $x_\al\to \xi$ in SOT topology and hence norm dense in $L_2(\M)$. For $L_1$, it suffices to show that $\A$ is $L_1$-norm dense in $L_1(\M)\cap \M$. Indeed, for any positive $\rho\in L_1(\M)\cap B$, we take $x_\al\to \rho^{1/2}$ in SOT topology and in $L_2(\M)\cap B$. Then for any subsequence $(x_n)\subset(x_\al)$,
  \begin{align*}\lim_{n\to \infty}\norm{x_n^*x_n-\rho}{1}\le&\norm{x_n^*x_n-\rho^{1/2}x_n}{1}+\norm{\rho^{1/2}x_n-\rho}{1}
  \\ \le & \lim_{n\to \infty}
  \norm{x_n^*-\rho^{1/2}}{2}\norm{x_n}{2}+\norm{\rho^{1/2}}{2}\norm{\rho^{1/2}-x_n}{2}= 0 \end{align*}
  Then $x_n^*x_n \to \rho \in L_1(\M)$ and $x_n^*x_n\in \A\cap B$ since $\A$ is a $*$-subalgebra.
\end{proof}

The next lemma shows that the relative entropy is continuous in $L_1$-norm for bounded invertible densities.

\begin{lemma}\label{continuity}Let $\rho\in S_B(\M)$ and $\rho_n$ be a sequence in $L_1(\M)_+$ such that $\norm{\rho_n-\rho}{1}=0$.
Suppose there exist $m,M>0$ such that $m 1\le \rho_n\le M 1$ for any $n$.
Then
$\displaystyle \lim_{n\to \infty}H(\rho_n)=H(\rho)$ and $\displaystyle \lim_{n\to \infty}D(\rho_n||\N)=D(\rho||\N)$.
\end{lemma}
\begin{proof}
We assume that $m 1 \le \rho\le M 1$. The lower semi-continuity inherited from relative entropy,
\[ H(\rho)=D(\rho|| 1)\le \liminf_{n\to \infty}D(\rho_n || 1)=\liminf_{n\to \infty}H(\rho)\]
For the upper continuity, we use Klein's inequality \cite[Theorem 5.9]{wirth} for $h(s)=s\log s$
\[ H(\rho_n)-H(\rho)=\tau(h(\rho_n)-h(\rho))\le \tau(h'(\rho_n)(\rho_n-\rho)),\]
where $h'(s)=1+\log s$ is the derivative of $h$. Because $m 1\le \rho_n\le M 1$,  we have ${\norm{h'(\rho_n)}{\infty}\le \max\{\log M, -\log m\}+1}$ is uniform bounded for $n$. Thus\begin{align*}\limsup_{n\to \infty}H(\rho_n)-H(\rho)\le & \limsup_{n\to \infty}\tau(h'(\rho_n)(\rho_n-\rho))\\ \le &\limsup_{n\to \infty}\pl (\max\{\log M, -\log m\}+1)\norm{\rho_n-\rho}{1}=0\pl,
\end{align*}
which implies $\limsup_{n}H(\rho_n)\le \limsup_{n}H(\rho_n)=H(\rho)$. For $D(\rho||\N)$ we use the decomposition $D(\rho||\N)=H(\rho)-H(E(\rho))$.
Note that $m 1=m E(1)\le E(\rho)\le M E(1)=M 1$ and
\[\lim_{n}\norm{E(\rho_n)-E(\rho)}{1}\le \lim_{n} \norm{\rho_n-\rho}{1}=0\pl,\]
By the same argument, we obtain $H(E(\rho))=\lim_n H(E(\rho_n))$. \end{proof}
Now we can show that $\la$-MLSI inequality for density in $S_B(\A_0)$ is equivalent to entropy decay property for all density in $S(\M)$. Recall that $S_B(\A_0)=S_B(\M)\cap \A_0$ where $\A_0=\bigcup_{t>0} T_t(\A)\subset \dom(A)$. Note that by the continuity of $T_t$ on $L_1$ (see \cite[Proposition 2.14]{DL92}), the positive part $(\A_0)_+$ is norm dense in $\A_+$ hence by Lemma \ref{densestate} also dense in $L_1(\M)_+$. Moreover, since $\A_0$ is a linear subspace containing unit, $S_B(\A_0)$ is norm dense in $S(\M)$.
\begin{prop}\label{equiv}
A semigroup $T_t$ satisfies $\la$-MLSI if and only if
\[ D(T_t(\rho)||\N)\le e^{-2\la t} D(\rho||\N)\pl, \pl  \forall \pl \rho \in S(\M).\]
\end{prop}
\begin{proof}By the heuristic discussion and the equation \eqref{derivative}, we know that our Definition \ref{defMLSI} of $\la$-MLSI is equivalent to \[D(\rho_t||\N)\le e^{-2\la t}D(\rho||\N)\pl, \forall \rho\in S_B(\A_0)\pl.\]
To extend the exponential decay to all of $S(\M)$,
it suffices to show that for any $\rho\in S(\M)$, there exists a sequence of $\rho_n\in S_B(\A_0)$ such that
\begin{align} \label{approx}\rho_n\to \rho \pl \text{in weakly $L_1$} \pl , \lim_{n\to \infty}D(\rho_n||\N)=D(\rho||\N)\pl.\end{align}
This is because by the lower semicontinuity of relative entropy (c.f. \cite[Corollary 5.12]{petzbook}) w.r.t to $L_1$-norm,
\[ D(T_t(\rho)||\N)\le \liminf_{n}  D(T_t(\rho_n)||\N)\le \liminf_{n}  e^{-2\la t} D(\rho_n||\N)=e^{-2\la t} D(\rho||\N)\pl.\]
which implies the assertion. We verify the claim by two steps: (1) for any $\rho\in S(\M)$, there exists a sequence $\rho_n\in S_B(\M)$ satisfying \eqref{approx}; (2) for any $\rho\in S_B(\M)$, there exists a sequence $\rho_n\in S_B(\A_0)$ satisfying \eqref{approx}. We first proves (2). By Lemma \ref{densestate}, for $\rho\in S_B(\M)$ with $\rho\le M 1$, there exists a sequence $\rho_n\in S(\A)$ such that $\rho_n\to \rho$ in $L_1$ and $\rho_n\le M 1$. Since $T_{t_n}(\rho_n)\to \rho_n$ in $L_1$, we can assume $\rho_n\in S(\A_0)$ by replacing $\rho_n$ by $T_{t_n}(\rho_n)$ for some small $t_n$. For any $0<\epsilon<1$, we define
\[\rho_{n,\epsilon}=(1-\epsilon)\rho_n+\epsilon 1\pl, \rho_{\epsilon}=(1-\epsilon)\rho+\epsilon 1\pl.\]
Then for each $\epsilon$, we have $\rho_{n,\epsilon}\to \rho_{\epsilon}$ in $L_1$ and by Lemma \ref{continuity}, $\lim_{n}D(\rho_{n,\epsilon}||\N)= D(\rho_{\epsilon}||\N) $ because $\epsilon 1 \le\rho_{n,\epsilon}\le M 1$. Moreover, by convexity and lower semi-continuity \[\limsup_{\epsilon\to 0}D(\rho_{\epsilon}||\N)\le \limsup_{\epsilon\to 0} (1-\epsilon) D(\rho||\N)=D(\rho||\N)\le \liminf_{\epsilon\to 0} D(\rho_{\epsilon}||\N)\pl.\]
Thus $\displaystyle D(\rho||\N)= \lim_{\epsilon\to 0} D(\rho_{\epsilon}||\N)$ and this proves (2).
For (1), we denote $e_n$ as the spectral projection of $E(\rho)$ for the spectrum $[1/n,n]$ and $e_n^\perp=1-e_n$. Without losing generosity, we assume $\rho$ is faithful otherwise we restrict the discussion on its support. Note that $\norm{e_n^\perp}{1}=\tau(e_n^\perp)\to 0$.
For each $n$, we define CPTP map
\[P_n:L_1(\M)\to L_1(\M)\pl, P_n(x)=e_nxe_n+\tau(xe_n^\perp)1 \]
We have $P_n(L_1(\N))\subset L_1(\N)$ and hence by data processing
\begin{align}\label{3} D(P_n(\rho)||\N)\le D(\rho||\N)\pl, \pl \forall n\pl.\end{align}
On the other hand, $E(P_n(\rho))=e_nE(\rho)e_n+\tau(E(\rho) e_n^\perp)1 $ converges to $E(\rho)$ in $L_1$-norm and $P_n(\rho)\to \rho$ in weakly. Indeed, for any $y\in \M$
 \begin{align*}\lim_{n}|\tau(\rho  y)-\tau(e_n\rho e_n y)|\le &\lim_{n}|\tau(e^\perp_n\rho  y)|+|\tau(e_n\rho e_n^{\perp} y)|
 \\ \le & \lim_{n}\norm{e^\perp_n}{1}\norm{\rho}{1}\norm{y}{\infty}+
\norm{e^\perp_n}{1}\norm{\rho}{1}\norm{y}{\infty}= 0\pl.
 \end{align*}
 Thus by the lower semicontinuity again
 \[ D(\rho||\N)=D(\rho||E(\rho))\le \liminf_{n}  D(\rho_n||E(\rho_n))=D(\rho_n||\N)\pl.\]
 Combined with \eqref{3}, we have $\lim_{n}D(\rho_n||\N)=D(\rho||\N)$. That completes the proof.
\end{proof}

\bibliography{fdiv}
\bibliographystyle{plain}
\end{document}